\documentclass{amsart}

\usepackage[utf8]{inputenc}
\usepackage[T1]{fontenc}
\usepackage{amscd,amssymb,amsfonts,accents}
\usepackage{hyperref}
\usepackage{fullpage}

\usepackage{color}
\usepackage[mathscr]{euscript}
\usepackage[all]{xy}
\usepackage{fancyhdr}

\allowdisplaybreaks[3]

\numberwithin{equation}{section}

\newtheorem{prop}{Proposition}[section]
\newtheorem{lemma}[prop]{Lemma}
\newtheorem{thm}[prop]{Theorem}
\newtheorem{cor}[prop]{Corollary}

\theoremstyle{definition}
\newtheorem{defn}[prop]{Definition}
\newtheorem{const}[prop]{Construction}

\newtheorem{rmk}[prop]{Remark}
\newtheorem{ex}[prop]{Example}
\newtheorem{ass}[prop]{Assumption}

\DeclareMathOperator{\rep}{Rep}
\DeclareMathOperator{\tr}{Tr}    
\DeclareMathOperator{\spec}{Spec} 
\DeclareMathOperator{\sym}{Sym}
\DeclareMathOperator{\im}{Im}
\DeclareMathOperator{\Gal}{Gal}

\DeclareMathOperator{\proj}{Proj}
\DeclareMathOperator{\Stab}{Stab} 
\DeclareMathOperator{\diag}{diag}

\DeclareMathOperator{\Aut}{Aut}
\DeclareMathOperator{\crep}{\cR ep}
\DeclareMathOperator{\Hilb}{Hilb}

\DeclareMathOperator{\Map}{Map}
\DeclareMathOperator{\Ext}{Ext}

\newcommand{\Gbar}{\overline{\mathbf{G}}_{Q,d}}
\newcommand{\ModSirs}{\mathcal{M}_{Q,d}^{(\Si,\theta)-rs}}
\newcommand{\Mods}{\mathcal{M}_{Q,d}^{\theta-s}}

\newcommand{\ModSiurs}{\mathcal{M}_{Q,d}^{(\Si,u,\theta)-rs}}
\newcommand{\RepSirs}{(\Rep^{\Si})^{\chi_\theta-rs}}
\newcommand{\RepSiurs}{({_u}\Rep^{\Si})^{\chi_\theta-rs}}

\newcommand{\ra}{\rightarrow}

\def\cB{\mathcal B}
\def\cH{\mathcal H}
\def\cL{\mathcal L}
\def\cM{\mathcal M}\def\cN{\mathcal N}
\def\cR{\mathcal R}\def\cT{\mathcal T}

\def\AA{\mathbb A}\def\CC{\mathbb C}
\def\EE{\mathbb E}\def\GG{\mathbb G}\def\HH{\mathbb H}

\def\NN{\mathbb N}\def\PP{\mathbb P}
\def\QQ{\mathbb Q}\def\RR{\mathbb R}

\def\ZZ{\mathbb Z}

\def\fg{\mathfrak g}
\def\fk{\mathfrak u}\def\fl{\mathfrak l}

\def\fu{\mathfrak u}

\def\SL{\mathrm{SL}}

\newcommand{\Id}{\mathrm{Id}}

\newcommand{\Ad}{\mathrm{Ad}}

\newcommand{\Mat}{\mathrm{Mat}}
\newcommand{\Hom}{\mathrm{Hom}}

\def\GL{\mathbf{GL}} 

\newcommand{\Q}{\mathbb{Q}}

\newcommand{\N}{\mathbb{N}}
\newcommand{\A}{\mathbb{A}}
\newcommand{\Z}{\mathbb{Z}}

\newcommand{\lra}{\longrightarrow}
\newcommand{\lmt}{\longmapsto}
\newcommand{\Mod}{\mathcal{M}_{Q,d}^{\theta-ss}}
\newcommand{\Modgs}{\mathcal{M}_{Q,d}^{\theta-gs}}
\newcommand{\Rep}{\rep_{Q,d}}
\newcommand{\G}{\mathbf{G}_{Q,d}}
\newcommand{\pphi}{\varphi}

\newcommand{\Gm}{\mathbb{G}_{m}}
\newcommand{\Si}{\Sigma}
\newcommand{\si}{\sigma}

\newcommand{\kb}{\overline{k}}

\newcommand{\RepQbar}{\rep_{\overline{Q},d}}
\newcommand{\GQbar}{\mathbf{G}_{\overline{Q},d}}
\newcommand{\LieG}{\fg_{Q,d}}
\newcommand{\Reps}{\rep_{Q,d}^{\theta-s}}
\newcommand{\ga}{\gamma}

\newcommand{\ov}[1]{\overline{#1}}

\title[Group actions on quiver varieties]{Group actions on quiver varieties and applications}
\author{Victoria Hoskins}
\address{Freie Universit\"at Berlin, Arnimallee 3, Raum 011, 14195 Berlin, Germany.}
\email{hoskins@math.fu-berlin.de}
\author{Florent Schaffhauser}
\address{Departamento de Matem\'aticas, Universidad de Los Andes, Carrera 1 \#18A-12, 111 711 Bogot\'a, Colombia \& IRMA UMR 7501, Universit\'e de Strasbourg, 7 rue Ren\'e Descartes, 67084 Strasbourg, France.}
\email{schaffhauser@math.unistra.fr}


\begin{document}

\begin{abstract}
We study algebraic actions of finite groups of quiver automorphisms on moduli spaces of quiver representations. We decompose the fixed loci using group cohomology and we give a modular interpretation of each component. As an application, we construct branes in hyperk\"{a}hler quiver varieties, as fixed loci of such actions.
\end{abstract}

\maketitle

\tableofcontents

\section{Introduction}

For a quiver $Q$ and an algebraically closed field $k$, King \cite{king} constructs moduli spaces of semistable $k$-representations of $Q$ of fixed dimension vector $d \in \NN^V$ as a geometric invariant theory (GIT) quotient of a reductive group $\G$ acting on an affine space $\Rep$ with respect to a character $\chi_\theta$ determined by a stability parameter $\theta \in \ZZ^V$. The stability parameter also determines a slope-type notion of $\theta$-(semi)stability for $k$-representations of $Q$, which involves testing an inequality for all proper non-zero subrepresentations. This GIT construction gives moduli spaces  $\cM^{\theta-(s)s}_{Q,d}$ of $\theta$-(semi)stable $k$-representations of $Q$ of dimension $d$, which are both coarse moduli spaces (if we consider semistable representations up to S-equivalence). 

The study of quiver representations is extremely fruitful in numerous areas of mathematics, from representation theory and algebraic geometry to mathematical physics, and many interesting varieties and moduli spaces arise as quiver varieties: for example, minimal resolutions of Kleinian singularities, such as the Hilbert scheme of points in the plane, are Najakima quiver varieties via the McKay correspondence, instanton moduli spaces admit quiver descriptions via ADHM data, and certain parabolic (Higgs) bundle moduli spaces arise as quiver varieties via (hyper)polygon spaces (for example, see \cite{ginzburg} and the references therein).

In this paper, we define a finite group $\Aut(Q)$ of covariant and contravariant automorphisms of $Q$ in $\S$\ref{quiver_aut_section} and consider certain, naturally induced, algebraic actions of subgroups $\Sigma \subset \Aut(Q)$ on $\Mod$, when the stability parameter $\theta$ and dimension vector $d$ are compatible with the $\Si$-action in the sense of $\S$\ref{quiver_aut_section}. 
We restrict the $\Si$-action to $\Mods \subset \Mod$, in order to use the fact that the stabiliser of every GIT stable point in $\Rep$ is a diagonal copy of $\GG_m$, denoted $\Delta$, in $\G$ (\textit{cf.}\ Remark \ref{stabiliser_of_stable}) to decompose the $\Si$- fixed locus of $\Mods$ in terms of the group cohomology of $\Si$ with values in $\Delta$ or the (non-Abelian) group $\G$. In \cite{HS_galois}, we use similar methods to decompose the rational points of moduli spaces of quiver representations over a perfect field by using actions of the absolute Galois group and many arguments simplify due to Hilbert's 90th Theorem, thus the decomposition in the Galois case is much cleaner. 

As an application, we construct submanifolds of hyperk\"{a}hler quiver varieties with rich holomorphic and symplectic geometry (known as branes \textit{cf.}\ \cite{Kapustin_Witten}) as fixed loci of finite groups of quiver automorphisms. Furthermore, the geometry of these fixed loci can be studied using our decomposition result (Theorem \ref{decomp_thm_qaut_intro} below) and given a modular interpretation by Corollary \ref{modular_interp_of_Sigma_fixed_pts_as_Sigma_equiv_rep}. 

\subsection{The decomposition of the fixed locus}

Let us outline the key steps appearing in the decomposition of the $\Si$-fixed locus in $\Mods$. First, as the $\Si$-action on $\Mod$ can be induced by compatible $\Si$-actions on $\Rep$ and $\G$, we construct a morphism 
\[ f_\Sigma : \Rep^{\Si}/\!/_{\chi_\theta} \G^{\Si} \lra (\Mod)^{\Si} \]
 of $k$-varieties in Propositions \ref{prop constr f cov} and \ref{prop constr of f general auto gp}, where for a $k$-variety $X$ with a $\Si$-action, $X^{\Si}$ denotes the fixed locus. We define the $(\Si,\theta)$-regularly stable locus by $\Rep^{(\Si,\theta)-rs}:=\Rep^\Si \times_{\Rep} \Rep^{\theta-s}$ and let $f^{rs}_\Si : \ModSirs \lra \Mods$ denote the restriction of $f_\Si$ to the moduli space of regularly stable representations. In general, the morphism $f^{rs}_\Si$ is neither surjective nor injective (\textit{cf.}\ Examples \ref{ex f not surj cov} and \ref{ex f not inj cov}, respectively). However, its failure to be bijective can be described in terms of group cohomology: the non-empty closed fibres of $f^{rs}_\Si$ are in bijection with the kernel of the pointed map $H^1(\Si, \Delta(k)) \lra H^1(\Si, \G(k))$, by Proposition \ref{non_empty_fibres}. This pointed map appears in an exact sequence
 \begin{equation}\label{les}
\ldots \lra H^1(\Si, \Delta(k)) \stackrel{\alpha}{\lra} H^1(\Si, \G(k)) \lra  H^1(\Si, \Gbar(k)) \stackrel{\delta}{\lra} H^2(\Si, \Delta(k))
\end{equation} 
associated to the short exact sequence of groups $1 \lra \Delta \lra \G \lra \Gbar \lra 1$ and other terms in this sequence will appear in our decomposition.
 
The second step is to construct a so-called type map from the closed points of the $\Si$-fixed locus to the second group cohomology of $\Si$ with values in $\Delta(k)$
\[ \cT : \Mods(k)^{\Si} \lra H^2(\Si,\Delta(k))\]
(\textit{cf}.\ Construction \ref{constr_type_map_def} and also $\S$\ref{sec cov and contr case}) and show that the image of $f^{rs}_{\Sigma}(k)$ is contained in the preimage of the trivial element under the type map; however, this containment can be strict (\textit{cf.}\ Example \ref{ex f not surj cov}). 

The third step is to introduce the notion of a modifying family $u$ (\textit{cf.}\ Definition \ref{modifying_family_def}) in order to modify the action of $\Si$ on $\Rep$ and $\G$ such that the induced $\Si$-action on $\Mod$ coincides with the original action. We let ${_u}\Rep^{\Si}$ and $ {_u}\G^{\Si}$ denote the fixed loci of these modified $\Si$-actions determined by $u$. Then we construct a morphism
${_u}f_{\Si} : {_u}\Rep^{\Si}/\!/_{\chi_\theta}\ {_u}\G^{\Si} \lra (\Mod)^{\Si}$
of $k$-varieties (\textit{cf.}\ Theorem \ref{new_fSigma} for the case of a covariant group of quiver automorphisms) and show that the image of ${_u}f^{rs}_{\Sigma}$ (that is, the restriction of ${_u}f_{\Si}$ to the preimage of $\Mods$) is contained in the preimage under the type map of the cohomology class of a $\Delta(k)$-valued 2-cocycle $c_u$ naturally defined by $u$ (\textit{cf.}\ Theorem \ref{new_fSigma}).

Finally, we describe the modifying families $u$ using group cohomology to obtain the following decomposition of the $\Si$-fixed locus in $\Mods$.

\begin{thm}\label{decomp_thm_qaut_intro}
Let $\Si \subset  \Aut^+(Q)$ be a group of covariant automorphisms of $Q$ for which $\theta$ and $d$ are compatible. Then we have the following set-theoretic decomposition of the set of closed points of the $\Si$-fixed locus of $\Mods$:
\[ (\Mods)^{\Si}(k) \qquad = \bigsqcup_{\begin{smallmatrix}[c_u] \in \im \cT \\ \ov{[b]} \in H^1_{u}(\Si, \G(k)) / H^1(\Si,\Delta(k)) \end{smallmatrix}} \im {_{u^b}}f^{rs}_{\Si}(k)\]
where $u^b$ is the modifying family determined by $[b]$ and $u$ via Lemma \ref{lemma mod family 1cocycle}, and ${_{u^b}}f^{rs}_{\Si}$ is the morphism defined as in Theorem \ref{new_fSigma}. 

Moreover, the non-empty fibres of ${_{u^b}}f_{\Si}^{rs}(k)$ are in bijection with the pointed set $$\ker\big(H^1(\Si,\Delta(k)) \lra H^1_{u^b}(\Si,\G(k))\big).$$

Finally, if the index set of this decomposition on closed points is finite, then this decomposition is induced by a decomposition of $(\Mods)^{\Si}$ as a coproduct in the category of varieties. In particular, this is the case in the following two situations:
\begin{enumerate}
\renewcommand{\labelenumi}{(\roman{enumi})}
\item $k=\CC$ and $\Sigma \subset  \Aut^+(Q)$ is arbitrary,
\item $k$ is algebraically closed and $\Sigma \subset  \Aut^+(Q)$ is cyclic.
\end{enumerate}
\end{thm}

The proof will be given in Section \ref{proof_main_result}, after Theorem \ref{components_of_fibres_of_the_type_map}. This decomposition for fixed loci of quiver automorphisms groups is more complicated than the decomposition in \cite{HS_galois} of rational points via Galois group actions (as several Galois cohomology groups in \cite{HS_galois} are trivial by Hilbert's 90th Theorem). In the quiver automorphism setting, the images of the morphisms ${_u}f_{\Si}$ may be strictly contained in $\cT^{-1}([c_u])$ and these morphisms may not be injective (\textit{cf.}\ Examples \ref{ex f not surj cov} and \ref{ex f not inj cov} respectively), but these issues do not arise in \cite{HS_galois}.

If $k$ is not algebraically closed, then one has a decomposition of the geometric points of the $\Sigma$-fixed locus in the moduli space $\Modgs$ of $\theta$-geometrically stable quiver representations (\textit{cf.}\ Remark \ref{rmks on decomp thm}). With some work, it may be possible to combine \cite{HS_galois} with the results in this paper, to obtain a decomposition of the rational points $(\Modgs)^{\Sigma}(k)$ for a perfect field $k$.

We provide a modular interpretation of this decomposition by observing that the domains of the morphisms ${_{u^b}}f^{rs}_{\Si}$ can also be described as moduli spaces of so-called $(\Si,u^b)$-equivariant representations of $Q$ in Corollary \ref{modular_interp_of_Sigma_fixed_pts_as_Sigma_equiv_rep}. For the trivial modifying family $u = 1$, the domain of $f_{\Si}$ can also be described as a moduli space of representations of a quotient quiver $Q/\Si$ (\textit{cf.}\ Definition \ref{defn quotient quiver} and Corollary \ref{modular_interp_of_Si_fixed_reps}).

\subsection{Constructing branes in hyperk\"{a}hler quiver varieties }
One can construct an algebraic symplectic analogue of $\Mod$ as a moduli space $\cN_{\overline{Q}}$ of representations of a doubled quiver $\overline{Q}$, modulo relations defined by a moment map. Over the complex numbers, if $\cN_{\overline{Q}}$ is smooth, it is hyperk\"{a}hler; for example, Nakajima quiver varieties can be described in this way (see \cite[p. 261]{CBmoment}). 

In $\S$\ref{branes_section}, we construct submanifolds of algebraic symplectic (and hyperk\"{a}hler, when $k = \CC$) quiver varieties as fixed loci for actions of subgroups of $\Aut(\overline{Q})$ (and for complex conjugation, when $k = \CC$) with rich symplectic (and holomorphic) geometry. Such submanifolds can be described in the language of branes \cite{Kapustin_Witten}, which will be recalled in Section \ref{branes_section}. 
The study of branes in Nakajima quiver varieties was initiated in \cite{fjm}, where involutions such as complex conjugation, multiplication by $-1$ and transposition are used to construct branes in Nakajima quiver varieties. We add to this picture branes arising from actions of quiver automorphisms that are $\ov{Q}$-(anti)-symplectic (this is a combinatorial notion, \textit{cf.}\ Definition \ref{def anti sympl auto}).

\begin{thm}
Let $k = \CC$ and assume that $\cN_{\overline{Q}}$ is smooth, and thus hyperk\"{a}hler. Let $\sigma$ be an involution of $\overline{Q}$ and $\tau : \CC \lra \CC$ denote complex conjugation; then
\begin{enumerate}
\item $(\cN_{\overline{Q}})^\sigma$ is a $BBB$-brane (resp. $BAA$-brane) if $\sigma$ is $\ov{Q}$-symplectic (resp. $\ov{Q}$-anti-symplectic);
\item $(\cN_{\overline{Q}})^{\tau \circ \sigma}$ is an $ABA$-brane (resp. $AAB$-brane), if $\sigma$ is $\ov{Q}$-symplectic (resp. $\ov{Q}$-anti-symplectic).
\end{enumerate} 
\end{thm}

Interestingly, we can construct hyperholomorphic branes ($BBB$-branes) from groups of $\ov{Q}$-symplectic automorphisms, that need not be of order 2 (\textit{cf.}\ Theorem \ref{thm BBB brane}).

\subsection{Connections with previous work}
Contravariant involutions of a quiver were studied by Derksen and Weyman \cite{Derksen_Weyman}, Zubkov \cite{Zubkov_I,Zubkov_II}, Shmelkin \cite{Shmelkin}, Bocklandt \cite{Bocklandt} and later by Young \cite{young}, where Young's motivation comes from physics and as an application he constructs orientifold Donaldson-Thomas invariants. In \cite{young}, the action of a contravariant involution is also modified using what is called a \lq duality structure', which corresponds to our notion of modifying families. 
Motivated by questions in representation theory, Henderson and Licata study  actions of so-called \lq admissible' covariant automorphisms on Nakajima quiver varieties of type A and prove a decomposition of the fixed locus \cite{Henderson_Licata}; however, they do not use group cohomology type techniques or see phenomena such as the morphisms ${_u}f_{\Si}$ failing to be injective in their setting (for example, compare \cite[Lemma 3.17]{Henderson_Licata} with Proposition \ref{non_empty_fibres}).

For moduli spaces of principal $G$-Higgs bundles ($G$ being a complex reductive group) over a smooth complex projective curve $X$, actions induced by involutions (or automorphisms) of $X$ and involutions of $G$ have been studied in \cite{BS2,BS1,BGP,GPR,HS,GPW,Sch_JSG,Sch_JDG,Sch_actions}; the approach in these papers is usually based on the gauge-theoretic, rather than algebraic, constructions of those moduli spaces. In this paper, we utilise the algebraic construction of quiver moduli spaces; however, the results we obtain are of a similar flavour.

\subsection{Generalisations to other moduli problems and GIT quotients}
Many of our techniques can be applied to study fixed loci of group actions on more general GIT quotients. However, for the results involving group cohomology, one would need to assume that the stabiliser group of all GIT stable points is a fixed subgroup (analogous to the fixed subgroup $\Delta \subset \G$), and one would need this fixed subgroup to be Abelian, in order to define the type map, as the second cohomology is only defined for an Abelian coefficient group. Over the complex numbers, by the Kempf-Ness Theorem, one can relate GIT quotients with symplectic reductions. In the symplectic setting, by using moment maps, one could also perform an analogous study of such fixed loci.

\subsection{The structure of the paper} 
In $\S$\ref{quiver_moduli_sect}, we recall King's construction of moduli spaces of quiver representations. In  $\S$\ref{quiver_aut_section}, we study actions of quiver automorphism groups and prove our decomposition for the fixed loci. In $\S$\ref{branes_section}, we apply our results to construct branes in hyperk\"{a}hler quiver varieties, and in $\S$\ref{sec apps and exs}, we study some examples.

\subsection*{Notation} A quiver $Q=(V,A,h,t)$ is an oriented graph, consisting of a finite vertex set $V$, a finite arrow set $A$, and head and tail maps $h,t:A\lra V$. 

The main results are for quiver representations over an algebraically closed field; however, we discuss the case of non-algebraically closed fields in Remark \ref{rmks on decomp thm}.\eqref{non-closed field}. Therefore, unless otherwise specified, $k$ is an algebraically closed field, all schemes are of finite type over $k$ and by a point of such a scheme, we mean a closed point.

\subsection*{Acknowledgements.} The authors thank the Institute of Mathematical Sciences of the National University of Singapore, where part of this work was carried out, for their hospitality in 2016, and acknowledge the support from U.S. National Science Foundation grants DMS 1107452, 1107263, 1107367 "RNMS: Geometric structures And Representation varieties" (the GEAR Network). 

V. Hoskins is supported by the \textit{Excellence Initiative of the DFG at the Freie Universit\"{a}t Berlin}. 

F. Schaffhauser is supported by \textit{Convocatoria 2018-2019 de la Facultad de Ciencias (Uniandes), Programa de investigaci\'on “Geometr\'ia y Topolog\'ia de los Espacios de M\'odulos”}, the \textit{European Union’s Horizon 2020 research and innovation programme under grant agreement No 795222} and the \textit{University of Strasbourg Institute of Advanced Study (USIAS)}.

We thank the referee for their detailed reading and many valuable suggestions, as well as for pointing out additional references to previous work on the subject.

\section{Moduli spaces of quiver representations}\label{quiver_moduli_sect}

Let us recall King's construction \cite{king} of moduli spaces of representations of a quiver $Q=(V,A,h,t)$ over an algebraically closed field $k$.

\begin{defn}\label{k_rep_def}
A $k$-representation of $Q$ is a tuple $W:=((W_v)_{v\in V}, (\pphi_a)_{a\in A})$ where:
\begin{itemize}
\item $W_v$ is a finite-dimensional $k$-vector space for all $v\in V$;
\item  $\pphi_a: W_{t(a)}\lra W_{h(a)}$ is a $k$-linear map for all $a\in A$.
\end{itemize} 
\end{defn}

\noindent The dimension vector of $W$ is the tuple $d=(\dim_k W_v)_{v\in V}$; we then say $W$ is $d$-dimensional. There are also natural notions of morphisms of quiver representations and subrepresentations.

\subsection{The GIT construction}
Every $k$-representation of $Q$ of fixed dimension vector $d=(d_v)_{v\in V}\in \N^V$ is isomorphic to a point of the following affine $k$-space
\[ \Rep := \prod_{a\in A} \Mat_{d_{h(a)}\times d_{t(a)}}.\]
The reductive $k$-group $\G:=\prod_{v\in V} \GL_{d_v}$ acts algebraically on $\Rep$ by conjugation: for $g = (g_v)_{v \in V} \in \G$ and $M = (M_a)_{a \in A} \in \Rep$, we have
 \begin{equation}\label{action_of_G_on_Rep}
g\cdot M := (g_{h(a)} M_a g_{t(a)}^{-1})_{a\in A},
\end{equation}
and the orbits for this action are in bijection with the set of isomorphism classes of $d$-dimensional $k$-representations of $Q$. 
One would like to construct a moduli space for quiver representations as a quotient of this action; however, by a results of Le Bruyn and Procesi, $k[\Rep]^{\G}$ is generated by traces of oriented cycles in $Q$ and so the affine GIT quotient is a point if $Q$ has no oriented cycles \cite{lBP}.

Instead King constructs a GIT quotient of the $\G$-action on an open subset of $\Rep$ by linearising the action using a stability parameter $\theta=(\theta_v)_{v\in V}\in \Z^V$. Let  $\theta':=(\theta'_v)_{v\in V}$ where $\theta'_v := \theta_v \sum_{\alpha\in V} d_\alpha - \sum_{\alpha\in V} \theta_\alpha d_\alpha$ for all $v\in V$; then $\sum_{v\in V} \theta'_v d_v =0$ and $\theta$ determines a character $\chi_\theta: \G \lra \GG_m$
\begin{equation}\label{the_character}
\chi_{\theta}( (g_v)_{v \in V} ):=\prod_{v\in V} (\det g_v)^{-\theta'_v}.
\end{equation} 
We let $\cL_\theta$ denote the $\G$-linearisation on the trivial line bundle $\Rep\times\A^1$, where $\G$ acts on $\AA^1$ via multiplication by $\chi_\theta$. As explained in \cite{king}, the invariant sections of positive powers $\cL_\theta^n$ of this linearisation are functions $f:\Rep\lra k$ satisfying  $f(g\cdot M) = \chi_\theta(g)^n\, f(M)$, for all $g\in\G$ and all $M\in\Rep$. We note that, as the subgroup of $\G$
 \begin{equation}\label{Delta_subgp}
\Delta:= \{ (t I_{d_v})_{v\in V} : t\in \GG_m\} \cong \GG_m,
\end{equation} 
acts trivially on $\Rep$, invariant sections of $\cL_\theta^n$ only exist if $\chi_\theta(\Delta) = 1$; by construction of $\theta'$, this holds, as $\sum_{v\in V} \theta'_v d_v =0$. 

The $\G$-invariant sections of positive powers of $\cL_\theta$ are used to determine GIT notions of (semi)stability with respect to $\chi_\theta$ (see, \cite[Definition 2.1]{king}, where we note that the notion of stability is modified to account for the presence of the global stabiliser $\Delta$). This determines open subsets $\Rep^{\chi_\theta-(s)s}$ of $\chi_\theta$-(semi)stable points and there is a GIT quotient
\[ \pi : \Rep^{\chi_\theta-ss} \lra \Rep/\!/_{\chi_\theta} \G:= \proj \bigoplus_{n \geq 0} H^0(\Rep, \cL_\theta^{n})^{\G}, \]
which is a good quotient of the $\G$-action on $\Rep^{\chi_\theta-ss}$ and restricts to a geometric quotient $\pi|_{\Rep^{\chi_\theta-s}} : \Rep^{\chi_\theta-s} \lra \Rep^{\chi_\theta-s}/\G$ of the GIT stable set.

\subsection{Slope semistability and moduli spaces}

For a stability parameter $\theta$, King proves that GIT $\chi_\theta$-(semi)stability in $\Rep$ can be reinterpreted as a slope type notion of $\theta$-(semi)stability for quiver representations in order to construct moduli spaces of $\theta$-semistable representations of $Q$. We note that King's original formulation of $\theta$-(semi)stablility involved checking the positivity of a certain integer, but we use the equivalent reformulation as a slope-type condition.

\begin{defn}[(Semi)stability] The $\theta$-slope of a $k$-representation $W$ of $Q$ is $$\mu_\theta(W) := \frac{\sum_{v\in V} \theta_v\dim_k W_v}{\sum_{v\in V} \dim_k W_v} \in \Q.$$
We say $W$ is:
\begin{enumerate}
\item $\theta$-semistable if $\mu_\theta(W') \leq \mu_\theta(W)$ for all $k$-subrepresentation $0 \neq W'\subset W$.
\item $\theta$-stable if $\mu_\theta(W') < \mu_\theta(W)$ for all $k$-subrepresentation $0 \neq W'\subsetneq W$.
\end{enumerate}
There are natural notions of Jordan-H\"{o}lder and Harder--Narasimhan filtrations, and we say two $\theta$-semistable $k$-representations of $Q$ are $S$-equivalent if their associated graded objects for their Jordan--H\"{o}lder filtrations are isomorphic.
\end{defn}

Using the Hilbert--Mumford criterion, King shows the GIT and slope notions of semistability on $\Rep$ coincide, and they can both be described using 1-parameter subgroups (1-PS) of $\G$. Consequently, he constructs moduli spaces of $\theta$-(semi)stable quiver representations using the GIT quotient described above.

\begin{thm}\cite{king}\label{GIT_const_of_Mod}
A point of $\Rep$ is GIT $\chi_\theta$-(semi)stable if and only if it is $\theta$-(semi)stable as a $k$-representation of $Q$. Moreover,
\[ \Mod:= \Rep /\!/_{\chi_\theta} \G \quad   (\text{resp. } \Mods:= \Rep^{\chi_\theta-s}/\G) \]
is a coarse moduli space for $S$-equivalence (resp. isomorphism) classes of $\theta$-semi\-stable (resp. $\theta$-stable) $d$-dimensional $k$-representations of $Q$.
\end{thm}

Let us end with an important observation that will be repeatedly used in $\S$\ref{quiver_aut_section}. 

\begin{rmk}\label{stabiliser_of_stable}
For a point $M \in \Rep^{\theta-s}$, we have 
\[\Stab_{\G}(M) = \Delta(k) \subset \G(k),\]
as the automorphism group of any stable representation is isomorphic to $\GG_m$.
\end{rmk}

\section{Automorphisms of quivers}\label{quiver_aut_section}

\begin{defn}
A covariant (resp.\ contravariant) automorphism of $Q=(V,A,h,t)$ is a pair of bijections $(\si_V:V\lra V,\si_A:A\lra A)$ such that
\begin{equation}\label{compatibility of h and t for cov autos}
\begin{array}{cccl}
t \circ \sigma_A = \sigma_V \circ t & \text{and} & h \circ \sigma_A = \sigma_V \circ h & \quad \text{if } \sigma \text{ is covariant,}\\
t \circ \sigma_A = \sigma_V \circ h & \text{and} & h \circ \sigma_A = \sigma_V \circ t & \quad \text{if } \sigma \text{ is contravariant.} 
\end{array}
\end{equation}
Henceforth, to simplify notation, we denote $\sigma_V$ and $\sigma_A$ both by $\sigma$. We let $\Aut^+(Q)$ (resp.\ $\Aut^-(Q)$) denote the group (resp.\ the set) of covariant (resp.\ contravariant) automorphisms of $Q$ and write $\Aut(Q) := \Aut^+(Q) \sqcup \Aut^-(Q)$.
\end{defn}

Note that, for any field $k$, a covariant (resp.\ contravariant) automorphism $\sigma$ of $Q$ determines a graded algebra automorphism (resp.\ anti-automorphism) of the path algebra $kQ$. We refer to the covariant automorphism of $Q$ given by $\si_V=\Id_V$ and $ \si_A=\Id_A$ as the trivial automorphism of $Q$. As the composition of covariant automorphisms is covariant, $\Aut^+(Q)$ is a subgroup of $\Aut(Q)$.
There is a group homomorphism $\text{sign} : \Aut(Q) \lra \{ \pm 1 \}$ sending $\si$ to $-1$ if and only if $\si$ is contravariant. Evidently, $\ker(\text{sign})=\Aut^+(Q)$.

For a subgroup $\Sigma \subset  \Aut(Q)$ of quiver automorphisms, we want to study induced actions of $\Sigma$ on the moduli space $\Mod$ of $\theta$-semistable $d$-dimensional $k$-representations of $Q$ (provided $d$ and $\theta$ are $\Sigma$-compatible in the sense of Definitions \ref{Si_comp_dim_vector_and_stab_param} and \ref{Si_comp_dim_vector_and_stab_param_general}). For a subgroup $\Sigma \subset  \Aut(Q)$, either $\Sigma \subset  \Aut^+(Q)$ or $\Sigma$ is an extension
\[ 1 \lra \Sigma^+ \lra \Sigma \lra \{ \pm 1 \} \lra 1, \]
where $\Sigma^+ \subset \Aut^+(Q)$ is a subgroup of covariant automorphisms. Therefore, one can start by studying the actions by subgroups of covariant automorphisms and actions by contravariant involutions. Since contravariant automorphisms of order two are studied in \cite{Derksen_Weyman,Shmelkin,Zubkov_I,Zubkov_II,Bocklandt,young}, we restrict our attention to subgroups $\Sigma \subset  \Aut^+(Q)$ of covariant automorphisms until $\S$\ref{sec cov and contr case}

\subsection{The induced action on the moduli space}\label{sec action quiver autos}

Let $\Sigma \subset  \Aut^+(Q)$ be a subgroup of covariant automorphisms and let $k$ be a field. In this section, we will construct algebraic $\Sigma$-actions on moduli spaces $\Mod$ of $\theta$-semistable $d$-dimensional $k$-representations of $Q$, when $d$ and $\theta$ are $\Sigma$-compatible in the following sense.

\begin{defn}\label{Si_comp_dim_vector_and_stab_param}
Let us denote $\si(d):=(d_{\si(v)})_{v\in V}$ and $\si(\theta):=(\theta_{\si(v)})_{v\in V}$.
\begin{enumerate}
\item A dimension vector $d=(d_v)_{v\in V}$ is $\Sigma$-compatible if $\sigma(d) = d$ for all $\sigma \in \Sigma$.
\item A stability parameter $\theta=(\theta_v)_{v\in V}$ is $\Sigma$-compatible if $\sigma(\theta) = \theta$ for all $\sigma \in \Sigma$.
\end{enumerate}
\end{defn}

\noindent Throughout the rest of Section \ref{quiver_aut_section}, we assume that $d$ and $\theta$ are $\Sigma$-compatible. 
To construct the $\Sigma$-action on $\Mod$, we use the GIT construction of this moduli space (Theorem \ref{GIT_const_of_Mod}). As $d$ is $\Si$-compatible, we have algebraic actions 
\begin{equation}\label{cov action}
\Phi: \Sigma   \times \Rep \lra \Rep \quad \text{and} \quad \Psi: \Sigma \times \G \lra \G 
\end{equation}
given by, for $\sigma \in \Sigma$ and for $M:=(M_a)_{a\in A} \in \Rep$ and $g:=(g_v)_{v \in V} \in \G$, 
\[ \Phi_\sigma(M):= \sigma(M):=(M_{\sigma(a)})_{a \in A} \quad \text{and} \quad 
 \Psi_\sigma(g):=\sigma(g):=(g_{\sigma(v)})_{v \in V}.\]
These actions are compatible with each other in the sense that
\begin{equation}\label{cov action comp}
\Phi_\sigma(g \cdot M) = \Psi_\sigma(g) \cdot \Phi_\sigma(M).
\end{equation}

\begin{prop}\label{prop induced action on quotient in covariant setting}
For $\Sigma$-compatible $d$ and $\theta$, the following statements hold.
\begin{enumerate}
\item The $\Sigma$-action on $\Rep$ preserves the GIT (semi)stable sets $\Rep^{\chi_\theta-(s)s}$;
\item There is an induced algebraic $\Sigma$-action on the moduli spaces $\mathcal{M}^{\theta-(s)s}_{Q,d}$.
\end{enumerate}
\end{prop}
\begin{proof}
For (1), it suffices to check that, for all $\sigma \in \Sigma$, the image of each closed point in $\Rep^{\chi_\theta-(s)s}$ under $\Phi_\sigma$ lies in $\Rep^{\chi_\theta-(s)s}$. The $\Si$-compatibility of $\theta$ implies that $\chi_\theta$ is $\Sigma$-invariant and so the $\Si$-action preserves the invariant sections of powers of  $\cL_\theta$ as in \cite[Proposition 3.1]{HS_galois}, which proves (1). Equivalently, by the Hilbert--Mumford criterion, it suffices to check that for every $\theta$-(semi)stable $k$-representation $W = ( (W_v)_{v \in V}, (\varphi_a)_{a \in A})$, the $k$-representation $\sigma(W):=( ( W_{\sigma(v)})_{v \in V}, (\varphi_{\sigma(a)})_{a \in A})$ is also $\theta$-(semi)stable, which we show in Lemma \ref{theta_semistability_preserved}. 

For (2), recall that $\pi:\Rep^{\chi_\theta-(s)s} \lra \mathcal{M}^{\theta-(s)s}_{Q,d}$ is a categorical quotient of the $\G$-action. To define the induced action $\Phi' : \Sigma \times \mathcal{M}^{\theta-(s)s}_{Q,d} \lra \mathcal{M}^{\theta-(s)s}_{Q,d}$, we observe that, for each $\sigma \in \Sigma$, the morphism $\pi\circ\Phi_\si$ is $\G$-invariant by \eqref{cov action comp}, and so there is a unique morphism $\Phi_\sigma'$ making the following diagram commute
\begin{equation}\label{induced_action_on_quotient}
\xymatrix@1{\Rep^{\chi_\theta-(s)s} \ar[r]^{\Phi_\sigma} \ar[d]^{\pi} &\Rep^{\chi_\theta-(s)s}
\ar[d]^{\pi} \\
\mathcal{M}^{\theta-(s)s}_{Q,d} \ar[r]^{\Phi'_\sigma} & \mathcal{M}^{\theta-(s)s}_{Q,d},}
\end{equation} given by the universal property of the categorical quotient $\pi$.
\end{proof}

\subsection{Morphisms to the fixed-point set of the action}

For $\Si \subset \Aut^+(Q)$, we let  $(\Mod)^{\Sigma}$ denote the fixed locus of the above action, which is a closed $k$-subvariety of $\Mod$. In order to describe this fixed locus, we construct morphisms from related moduli spaces to this fixed locus in this section.

\begin{defn}[Quotient quiver]\label{defn quotient quiver} Given $\Sigma \subset \Aut^+(Q)$, we define the quotient quiver $$Q/\Sigma:=(V/\Sigma,A/\Sigma, \tilde{h},\tilde{t}),$$ where $V/\Sigma$ and $A/\Sigma$ denote the set of $\Sigma$-orbits in $V$ and $A$ respectively, and the head and tail maps $\tilde{h},\tilde{t}: A/\Sigma \lra V/\Sigma$ are given by
$\tilde{h}( \Sigma \cdot a) = \Sigma \cdot h(a)$ and $\tilde{t}( \Sigma \cdot a) = \Sigma \cdot t(a)$,
which are well-defined by (\ref{compatibility of h and t for cov autos}). 
\end{defn}

\noindent A $\Sigma$-compatible dimension vector $d$ and stability parameter $\theta$ for $Q$ determine a dimension vector $\tilde{d}$ and stability parameter $\tilde{\theta}$ for $Q/\Sigma$, where $\tilde{d}_{\Sigma \cdot v} := d_v$ and $\tilde{\theta}_{\Sigma \cdot v} := |\Sigma \cdot v|\theta_v$. Let $\G^{\Si}$ and $\Rep^{\Si}$ denote the fixed loci for the 
actions of $\Si$ on $\G$ and $\Rep$ given by $\Psi$ and $\Phi$ respectively; 
then the action of $\G^{\Si}$ on $\Rep$ preserves $\Rep^{\Si}$ by \eqref{cov action comp}. By the following lemma, $\G^{\Si}$ is reductive and so we can consider GIT quotients of the $\G^{\Si}$-action on $\Rep^{\Si}$.

\begin{prop}\label{lemma iso GQ with G for Q mod Sigma}
There are isomorphisms
\[ \alpha: \mathbf{G}_{Q/\Si, \tilde{d}} \overset{\simeq}{\lra} \G^{\Si} \quad \text{and} \quad \beta : \rep_{Q/\Si,\tilde{d}} \overset{\simeq}{\lra} \Rep^{\Si} \]
such that, if the two groups are identified through $\alpha$, then $\beta$ is equivariant with respect to the $\mathbf{G}_{Q/\Si, \tilde{d}}$-action on $ \rep_{Q/\Si,\tilde{d}}$ and the $\G^{\Si}$-action on $\Rep^{\Si}$.
\end{prop}
\begin{proof}
The isomorphisms are defined by 
\[ \alpha( (g_{\Si \cdot v})_{{\Si \cdot v} \in V/\Si}) = (g_{\Si \cdot v})_{v \in V} \quad \text{and} \quad  \beta((M_{\Si \cdot a})_{{\Si \cdot a} \in A/\Si}) = (M_{\Si \cdot a})_{a \in A},\] 
which are compatible with the actions by construction.
\end{proof}

\begin{cor}\label{modular_interp_of_Si_fixed_reps}
There are isomorphisms
$$\Rep^{\Si}/\!/_{\chi_\theta}\G^{\Si} \cong \rep_{Q/\Si,\tilde{d}}/\!/_{\chi_{\tilde{\theta}}} \mathbf{G}_{Q/\Si, \tilde{d}}$$ and $(\Rep^{\Si})^{\chi_\theta-s} / \G^{\Si} \cong \rep_{Q/\Si,\tilde{d}}^{\chi_{\tilde{\theta}}-s} / \mathbf{G}_{Q/\Si, \tilde{d}}$
where, on the left hand side of these isomorphisms, $\chi_\theta$ denotes the restriction to $\G^{\Si}$ of the character $\chi_\theta$ of $\G$.\\ Consequently, the GIT quotient $\Rep^{\Si}/\!/_{\chi_\theta}\G^{\Si}$ (resp.\ $(\Rep^{\Si})^{\chi_\theta-s} /\G^{\Si}$) can be interpreted as the moduli space
$\mathcal{M}_{Q/\Si,\tilde{d}}^{\tilde{\theta}-(s)s}$ of $\tilde{\theta}$-semistable (resp.\ $\tilde{\theta}$-stable) $\tilde{d}$-dimen\-sional representations of $Q/\Si$.
\end{cor}

\begin{proof}
The isomorphisms are a consequence of Proposition \ref{lemma iso GQ with G for Q mod Sigma}, the universal property of the categorical quotient and the observation that $\chi_{\tilde{\theta}} = \chi_\theta \circ \alpha$ (since, by definition, $\tilde{\theta}_{\Sigma \cdot v} = |\Sigma \cdot v|\theta_v$). The last assertion then follows from Theorem \ref{GIT_const_of_Mod}.
\end{proof}

\noindent In Corollary \ref{modular_interp_of_Sigma_fixed_pts_as_Sigma_equiv_rep}, we will provide another modular interpretation of the GIT quotient $\Rep^{\Si} /\!/_{\chi_\theta}\G^{\Si}$. 

For $\Si \subset \Aut^+(Q)$, we note that $\Si$-compatibility of $\theta$ is equivalent to $\chi_\theta$ being $\Si$-invariant.

\begin{lemma}\label{lemma compare ss}
Let $\Si  \subset \Aut^+(Q)$ be a finite group acting on $\Rep$ and $\G$ as defined at \eqref{cov action} and fix a stability parameter $\theta$ such that $\chi_\theta$ is $\Si$-invariant. Then, for a $\Si$-fixed point $M \in \Rep(k)$, the following statements are equivalent.
\begin{enumerate}
\item $M$ is GIT semistable for the $\G$-action on $\Rep$ with respect to the character $\chi_\theta : \G\lra \GG_m$.
\item $M$ is GIT semistable for the $\G^{\Si}$-action on $\Rep^{\Si}$ with respect to the restricted character $\chi_\theta : \G^{\Si} \lra \GG_m$.
\end{enumerate} 
\end{lemma}

\begin{proof}
Evidently, (1) implies (2). For the converse, we proceed by contrapositive and argue along the lines of \cite[Theorem 4.2]{Kempf}, which is about Galois actions but carries over naturally to our setting. So let us assume that $M\in\Rep^\Si(k)$ is not $(\G,\chi_\theta)$-semistable. Let $\Lambda_M$ denote the set of 1-PSs $\lambda$ of $\G$ for which the morphism $\GG_m\lra \Rep$ given by the $\lambda$-action on $M$ extends to $\AA^1$. By \cite[Theorem 3.4]{Kempf}, there is a canonical parabolic subgroup $P_M\subset \G$ such that $P_M = P(\lambda)$ for any 1-PS $\lambda \in \Lambda_M$ which minimises a normalised Hilbert--Mumford functional $a_M : \Lambda_M \lra \RR$ given by $a_M(\lambda) = \langle \chi_\theta, \lambda \rangle \! / || \lambda ||$, where $||-||$ is a length function on the 1-PSs of $\G$ in the sense of \cite[p. 305]{Kempf}. Since $M$ is fixed by $\Si$, the set of 1-PSs $\Lambda_M$ is $\Si$-invariant by the compatibility of the actions of $\Si$ and $\G$. By summing over $\Si$ (or applying Kempf's construction of length functions to the reductive group $\G\rtimes\Si$), we can assume that $|| - ||$ is $\Si$-invariant. Then, as $\chi_\theta$ is $\Si$-invariant, it follows that $a_M$ is $\Si$-invariant analogously to Lemma 4.1 in \textit{loc.\ cit}. Therefore, $P_M = \sigma(P_M)$ for all $\si \in \Si$, by the uniqueness of $P_M$ analogously to the proof of Part (b) of Theorem 4.2 in  \textit{loc.\ cit}. Hence, $P_M$ is a $\Si$-invariant subset; that is, $\Si \cdot P_M = P_M$. 

We can write $P_M = \Pi_{v \in V} P_{M,v}$, where $P_{M,v} \subset \GL_{d_v}$ is a parabolic subgroup for all $v \in V$; then $\Si \cdot P_M = P_M$ implies that $P_{M,v} = P_{M,v'}$ for any two vertices $v,v'$ in the same $\Si$-orbit. Then 
\[ P_M \cap \G^{\Si} = \prod_{\Si \cdot v \in V/\Si} P_{M,v} \subset \G^\Si = \prod_{\Si \cdot v \in V/\Si} \GL_{d_v} \]
is a proper parabolic subgroup, as $P_M \subset \G$ is a proper parabolic subgroup by assumption. In particular, $M$ is not $(\G^{\Si}, \chi_\theta)$-semistable, as any 1-PS $\lambda \in P_M \cap G^{\Si}$ is in both $\G^{\Si}$ and $P_M$ and thus destabilises $M$.
\end{proof}

\begin{prop}\label{prop constr f cov}
In $\Rep^\Si$, there is an equality of $k$-subvarieties
\[ (\Rep^{\Si})^{(\G^\Si,\chi_\theta)-ss} = \Rep^{\Si} \times_{\Rep} \Rep^{(\G,\chi_\theta)-ss}.\]
The closed immersion 
$(\Rep^{\Si})^{(\G^\Si,\chi_\theta)-ss} \hookrightarrow \Rep^{(\G,\chi_\theta)-ss}$
induces a morphism
\[ f_{\Sigma}: \Rep^{\Si} /\!/_{\chi_\theta}(\G^\Si)  \lra \Rep /\!/_{\chi_\theta} \G\]
whose image is contained in the $\Si$-fixed locus.
\end{prop}
\begin{proof}
The $k$-subvarieties $(\Rep^\Si)^{(\G^\Si,\chi_\theta)-ss}$ and $\Rep^{\Si} \times_{\Rep} \Rep^{(\G,\chi_\theta)-ss}$ are open in $\Rep^{\Si}$.  So, to show that they agree, it suffices to check the equality on closed points, for which we can use Lemma \ref{lemma compare ss}. Then the closed immersion 
\[(\Rep^{\Si})^{(\G^{\Si},\chi_\theta)-ss} \hookrightarrow \Rep^{(\G,\chi_\theta)-ss},\] 
induces the morphism $f_{\Si}$ via the universal property of the categorical quotient. It is straightforward to check that the image of $f_{\Sigma}$ is contained in the $\Si$-fixed locus.
\end{proof}

\noindent We can interpret $f_\Si$ as a morphism $\mathcal{M}_{Q/\Si,\tilde{d}}^{\tilde{\theta}-ss} \lra \Mod$ by Corollary \ref{modular_interp_of_Si_fixed_reps}. Let us now study properties of $f_{\Si}$ (or strictly speaking a restriction $f_{\Si}^{rs}$ of $f_\Si$ as introduced below) in terms of the group cohomology of $\Si$. We first introduce the notion of regularly stable for an arbitrary finite group $\Si$ acting algebraically on $\Rep$ and $\G$ in a compatible way (in the sense that \eqref{cov action comp} holds for these actions).

\begin{defn}[Regularly stable point]\label{GIT_rs}
Let $\Si$ be a finite group acting on $\Rep$ and $\G$ as defined at \eqref{cov action}. A point $M\in\Rep$ which is both $\Si$-fixed and $(\G,\chi_\theta)$-stable is called a $(\Si,\chi_\theta)$-regularly stable point of $\Rep$. Let $$\RepSirs:= \Rep^{\Si} \times_{\Rep} \Rep^{(\G,\chi_\theta)-s}$$ be the $\G^\Si$-invariant open subset of $\Rep^\Si$ whose points are both $\Si$-fixed and $(\G,\chi_\theta)$-stable. 
\end{defn}

As $(\G,\chi_\theta)$-stability implies $(\G^\Si,\chi_\theta)$-stability, we have 
\begin{equation}\label{rs_implies_stable}
\RepSirs \subset (\Rep^\Si)^{(\G^\Si,\chi_\theta)-s}.
\end{equation} However, the converse inclusion is not true in general (as Lemma \ref{lemma compare ss} does not hold for stable points). By \eqref{rs_implies_stable}, we have a geometric quotient
\[ \ModSirs:= \RepSirs/ \G^{\Si}=f^{-1}_\Si(\Mods), \]
which is open in $(\Rep^\Si)^{\chi_\theta-s}/ \G^\Si \simeq \mathcal{M}_{Q/\Si,\tilde{d}}^{\tilde{\theta}-s}$ (using Corollary \ref{modular_interp_of_Si_fixed_reps}) and that we call the moduli space of $(\Si,\theta)$-regularly stable representations of $Q$. Let
\[ f_{ \Si}^{rs}:=f_{\Si}|_{\ModSirs} : \ModSirs \lra (\Mods)^{\Si}\] 
denote the restriction of $f_{\Si}$ to the open subscheme $\ModSirs \subset \mathcal{M}_{Q/\Si,\tilde{d}}^{\tilde{\theta}-ss}$.

\begin{prop}\label{prop image frs closed}
The image of $f_{\Si}^{rs}$ is a closed subvariety of $(\Mods)^\Sigma$.
\end{prop}
\begin{proof}
Since the $\G$-action on $\Reps$ is closed, the image $\G \cdot \RepSirs$ of $\G \times \RepSirs$ under that action morphism is closed in $\Reps$. Consider the following commutative diagram
\[  \xymatrixcolsep{3pc} \xymatrix{ \RepSirs \ar[r] \ar@{->>}[d]_{\pi_\Sigma} \ar@{->>}[rd]_{\pi} & {\G \cdot \RepSirs}  \ar@{^{(}->}[r] \ar@{->>}[d]_{\pi} & \Reps  \ar@{->>}[d]_{\pi} \\ 
\ModSirs \ar@{->>}[r] 
\ar@/_2pc/[rr]^{f^{rs}_{\Si}}  & {\pi(\RepSirs)} \ar@{^{(}->}[r] & \Mods } \]
where $\pi$ (resp. $\pi_\Sigma$) denotes the GIT quotient for the $\G$-action on $\Rep$ (resp. $\G^{\Si}$-action on $\Rep^{\Si}$). Since the GIT quotient $\pi$ sends $\G$-invariant closed subvarieties of $\Reps$ to closed subvarieties of $\Mods$, we see that the lower right inclusion is a closed immersion. Clearly the lower left map is surjective, and so the image of $f^{rs}_{\Si}$ is the closed subvariety $\pi(\RepSirs) \subset \Mods$.
\end{proof}

\begin{prop}\label{non_empty_fibres}
Let $m=\G\cdot M \in \Mods(k)$. If non-empty, the fibre $ (f_{\Si}^{rs})^{-1}(m)$ is in bijection with the pointed set $$\ker\big(H^1(\Si,\Delta(k)) \lra H^1(\Si,\G(k))\big).$$
\end{prop}

\begin{proof}
As the proof is similar to \cite[Proposition 3.3]{HS_galois}, we just give an outline of the main steps. By the Hilbert--Mumford criterion, $M$ is a $\theta$-stable $k$-representation and we recall from Remark \ref{stabiliser_of_stable} that the stabiliser of $M$ for the $\G$-action is $\Delta$. We give the  fibre $ f^{-1}_{\Si}(m)$ the base point given by the element $\G^{\Si} \cdot M$ and we will construct a map
\[ \beta : f^{-1}_{\Si}(m) \lra \ker\big(H^1(\Si,\Delta(k)) \lra H^1(\Si,\G(k))\big)\]
by associating to each $m'=\G^{\Si} \cdot M' \in f^{-1}_{\Si}(m)$, a normalised $\Delta(k)$-valued 1-cocycle $\beta_{m'}$ on $\Si$ which splits over $\G(k)$. As $m' \in f^{-1}_{\Si}(m)$, we know that $M$ and $M'$ lie in the same $\G$-orbit; thus there exists $g \in \G(k)$ such that $g \cdot M' = M$. For $\si \in \Si$, one verifies that $\beta_{m'}(\si):=g \si(g^{-1})$ stabilises $M$ and thus lies in $\Delta(k)$. Moreover,
\begin{enumerate}
\renewcommand{\labelenumi}{(\roman{enumi})}
\item $\beta_{m'}(1_{\Si}) = 1_\Delta$, and
\item $\beta_{m'}(\si_1\si_2) = \beta_{m'}(\si_1)\si_1(\beta_{m'}(\si_2))$ for all $\si_1,\si_2 \in \Si$;
\end{enumerate}
hence, $\beta$ is a normalised $1$-cocycle, which is, by definition, split over $\G(k)$. It is a straightforward computation to check that $\beta$ is well-defined; that is, the cohomology class of the 1-cocycle $\beta_{m'}$ is independent of the choice of the representatives $M$ and $M'$ of the orbits, and of the choice of element $g \in \G(k)$ such that $g \cdot M' = M$.

Let us show that $\beta$ is surjective. If $\gamma$ is a normalised $\Delta(k)$-valued 1-cocycle on $\Si$ which splits over $\G(k)$, then there exists $g \in \G(k)$ such that $\gamma(\sigma)=g\si(g^{-1})\in\Delta(k)$ for all $\si \in \Si$. Since $\Delta(k)$ acts trivially on $M$, we have $\si(g^{-1}\cdot M) = g^{-1}\cdot M$ for all $\si \in \Si$, and so $m':= \G^{\Si} \cdot g^{-1} \cdot M \in f^{-1}_\Si(m)$ and $\beta_{m'} = \gamma$. 

To prove that $\beta$ is injective, suppose that the $\Delta(k)$-valued $1$-cocycle $\beta_{m'}$ splits over $\Delta(k)$; that is, there exists $a\in \Delta(k)$ such that $\beta_{m'}(\si):=g\si(g^{-1}) =a\si(a^{-1})$ for all $\si\in\Si$. Then
\begin{enumerate}
\renewcommand{\labelenumi}{(\roman{enumi})}
\item $\si(a^{-1}g) = a^{-1}g$ for all $\si\in\Si$ (i.e. $a^{-1}g \in \G^{\Si}(k)$),
\item $(a^{-1}g) \cdot M' = M$;
\end{enumerate}
hence, $G^{\Si}(k) \cdot M' = G^{\Si}(k) \cdot M$, which completes the proof. 
\end{proof}

\noindent Since $\Si$  acts trivially on $\Delta \subset\G$, we have $H^1(\Si,\Delta(k))\simeq H^1(\Si,k^\times)$, where the latter is computed with respect to the trivial action of $\Si$ on $k^\times$. The next result gives a sufficient geometric condition for $f_\Si^{rs}$ to be injective.

\begin{cor}\label{prop fSigma inj}
Suppose that for each $\si \in \Si$, there is a vertex $v_\si \in V$ such that $\si(v_\si)=v_\si$, then the morphism $f_{\Si}^{rs}$ is injective.
\end{cor}
\begin{proof}
As $k$ is algebraically closed, it suffices to check that  $f_{\Si}^{rs}$ is injective on closed points and for this we can use the description of the fibres given by Proposition \ref{non_empty_fibres}. Thus, it suffices to show that $H^1(\Si,\Delta(k)) \lra H^1(\Si,\G(k))$ is injective. By definition, an element in the kernel of that map is the cohomology class of a $1$-cocycle $\alpha: \Si \lra\Delta(k)$ of the form $\alpha(\si)=g\si(g^{-1})$ for some $g=(g_v)_{v\in V}\in\G(k)$. As $\alpha$ is $\Delta(k)$-valued, for each $\si \in \Si$, there exists $t_\sigma \in \GG_m(k)$ such that $(\alpha(\si)_v)_{v \in V} := (g_v g^{-1}_{\si(v)})_{v \in V} = (t_\sigma I_{d_v})_{v\in V}$. In particular, for $v = v_\si$, we have $t_\sigma I_{d_v} =g_{v} g^{-1}_{\si(v)} = g_{v} g^{-1}_{v} = I_{d_v}$; that is, $t_\sigma = 1$ for all $\sigma$. Therefore, every such element $\alpha$ is trivial by our assumptions on the $\Si$-action on $V$.
\end{proof}

Next, we can use a so-called type map (analogously to the type map for Galois actions in \cite[Proposition 3.6]{HS_galois}) to determine whether $f_\Si^{rs}$ is surjective.

\begin{const}[The type map]\label{constr_type_map_def}
Let us introduce a map
\[ \cT : (\Mods)^{\Si}(k) \lra H^2(\Si,\Delta(k)), \]
which we call the type map. Since it is analogous to the construction of the type map for Galois actions in \cite[Proposition 3.6]{HS_galois}, we only outline the construction.

Let $m = \G \cdot M \in (\Mods)^{\Si}(k)$; then as $m$ is preserved by $\Si$, we have that, for all $\si \in \Si$, there exists $u_\si \in \G(k)$ such that $u_\si \cdot \si(M) = M$. For $\si_1,\si_2 \in \Si$, one observes that the element $c_m(\si_1,\si_2):=u_{\si_1}\si_1(u_{\si_2}) u_{\si_1\si_2}^{-1}$ stabilises $M$ and thus lies in $\Delta(k)$ by Remark \ref{stabiliser_of_stable}, as $M$ is $\theta$-stable. To verify that $c_m$ is a $\Delta(k)$-valued 2-cocycle of $\Si$, it remains to check that for $\si_i \in \Si$ for $i=1,2,3$, we have
\[c_m(\si_1,\si_2)c(\si_1\si_2,\si_3) = \si_1(c_m(\si_2,\si_3))c_m(\si_1,\si_2\si_3),\]
which is a consequence of the fact that $\Delta(k)$ is a $\Si$-invariant central subgroup of $\G(k)$. We define $\cT(m):=[c_m]$ and leave it to the reader to check that $\cT$ is well-defined (that is, the cohomology class of this 2-cocycle is independent of the choice of representative $M$ of the orbit and the elements $u_\sigma$ such that $u_\sigma \cdot \sigma(M) = M$).
\end{const}

\begin{rmk}
The type map $\cT$ factors via the connecting homomorphism
\[ \delta : H^1(\Si, \Gbar(k)) \lra H^2(\Si, \Delta(k))  \]
in the long exact sequence \eqref{les}, where $\Gbar := \G/\Delta$. Indeed, for $m  \in (\Mods)^{\Si}(k)$ and a family ${u}_\si$ of elements in $\G$ indexed by $\si \in \Si$ as above, we claim that the corresponding family $\ov{u}_\si$ of elements in $\Gbar$ indexed by $\si \in \Si$ is a 1-cocycle. The cocycle condition $ \ov{u}_{\si_1\si_2} = \ov{u}_{\si_1}\si_1(\ov{u}_{\si_2})$ follows from the fact that $u_{\si_1}\si_1(u_{\si_2}) u_{\si_1\si_2}^{-1} \in \Delta(k)$. One can check that the cohomology class $[\ov{u}]$ is independent of the choices and so we have a well-defined map $\widetilde{\cT}:(\Mods)^{\Si}(k) \ra H^1(\Si, \Gbar(k))$ given by $m\lmt [\ov{u}]$. By definition of the connecting homomorphism $\delta$, one has $\delta\circ\widetilde{\cT} = \cT$.
\end{rmk}

Let us explain the relationship between the type map and the image of $ f_{\Si}^{rs}$.

\begin{prop}\label{prop_basic_type_map}
We have $(\im f_{\Si}^{rs})(k) \subset \cT^{-1}([1])$.
\end{prop}
\begin{proof}
If $m=\G \cdot M \in \im f_{\Si}^{rs}(k)$, then we can pick a representative $M \in \Rep^{\Si}(k)$ and let $u_\si:=1_{\G}$ for all $\si \in \Si$. Then $u_\si \cdot \si(M)=M$ and so $c_m = 1 \in H^2(\Si,\Delta(k))$ and $\cT(m) = [1]$. 
\end{proof}

\begin{ex}\label{ex H2 cyclic gp}
If $\Si$ is a cyclic group of order $n$, then $$H^2(\Si,\Delta(k)) \simeq H^2(\Z/n\Z, k^*) \simeq k^*/(k^*)^{(n)} \simeq \{1\}.$$
\end{ex}

In Example \ref{ex f not surj cov}, we will see that one can have $(\im f_{\Si}^{rs})(k) \neq \cT^{-1}([1])$. The map $f_\Si^{rs}$ is not surjective in general and, to account for that failure, we will alter the algebraic $\Si$-actions on $\Rep$ and $\G$ by using a modifying family of elements in $\G(k)$ in the following sense. We recall that $\Phi$ and $\Psi$ denote the original $\Si$-actions on $\Rep$ and $\G$, respectively, and that $\Gbar := \G /\Delta$.

\begin{defn}[Modifying family]\label{modifying_family_def}
A modifying family of elements in $\G(k)$ is a tuple $u:=(u_\si)_{\si\in\Si}$ 
of elements $u_\sigma \in \G(k)$ indexed by $\si\in \Si$ such that 
\begin{enumerate}
\item $u_{1_{\Si}} = 1_{\G(k)}$,
\item $\ov{u}:=(\ov{u}_\si)_{\si\in\Si}$ is a $\Gbar$-valued 1-cocycle.
\end{enumerate}
We let $c_u(\si_1,\si_2):=u_{\si_1} \Psi_{\si_1}(u_{\si_2}) u_{\si_1\si_2}^{-1} \in \G(k)$ denote the associated $\Delta(k)$-valued $2$-cocycle.
\end{defn}

\begin{lemma}\label{new_Sigma_action}
Let $u:=(u_\si)_{\si\in\Si}$ be a modifying family of elements in $\G(k)$ indexed by $\Si$. 
Then we can define modified $\Si$-actions 
\[\Phi^u: \Si \times \Rep \lra \Rep;\quad (\si,\phi) \lmt \Phi^u_\si(\phi):=u_\si \cdot \Phi_\si(\phi)\] 
and
\[\Psi^u: \Si \times\G \lra \G;\quad (\si,g) \lmt \Psi^u_\si(g):=\Ad_{u_\si}\, \Psi_\si(g)\] 
which are compatible in the sense of \eqref{cov action comp} and such that the induced 
$\Si$-action on $\Mod$ coincides with that of Proposition \ref{prop induced action on quotient in covariant setting}.
\end{lemma}

\begin{proof}
It is straightforward to check that $\Phi^u$ and $\Psi^u$ are compatible $\Si$-actions. By using the universal property of the categorical quotient, one proves there is an induced $\Si$-action on the quotient $\Mod=\Rep/\!/_{\chi_\theta}\G$; this coincides with the $\Si$-action defined in Proposition \ref{prop induced action on quotient in covariant setting}, as the following diagram commutes
\[\xymatrix{\Rep^{\chi_\theta-ss} \ar[r]^{\Phi_\sigma} \ar[d]^{\pi} & \Rep^{\chi_\theta-ss} \ar[d]^{\pi} \ar[r]^{ u_\si \cdot } & \Rep^{\chi_\theta-ss} \ar[dl]^{\pi}\\
\Rep/\!/_{\chi_\theta}\G \ar[r]^{\Phi'_\sigma} & \Rep/\!/_{\chi_\theta}\G}\]
following the commutativity of Diagram \eqref{induced_action_on_quotient}.
\end{proof}

For a modifying family $u$, let ${_u}\Rep^{\Si}$ and ${_u}\G^{\Si}$ denote the fixed loci for the $\Si$-actions $\Phi^u$ and $\Psi^u$. Then ${_u}\Rep^{\Si}$ is a closed subscheme of $\Rep$ 
and ${_u}\G^{\Si}$ is a closed subgroup of $\G$, and moreover, the ${_u}\G^{\Si}$-action preserves ${_u}\Rep^{\Si}$. 

\begin{lemma}\label{lemma gp fixed locus reductive}
For a modifying family $u =(u_\si)_{\si \in \Si}$ of elements in $\G(k)$, the fixed locus for the $u$-modified $\Si$-action ${_u}\G^{\Si}$ is a smooth connected reductive group.
\end{lemma}
\begin{proof}
The fixed locus  ${_u}\G^{\Si}$ is the subgroup of elements $g = (g_v)_{v \in V}$ such that
$ g_v = u_{\sigma,v} g_{\sigma(v)} u_{\sigma,v}^{-1} $ for all $\sigma \in \Sigma$.
If we pick representatives $v_1, \dots , v_r$ of the $\Sigma$-orbits in $V$, then 
\[ {_u}\G^{\Si} \cong \prod_{i = 1}^r C_{\GL_{d_{v_i}}}(\{u_{\sigma,v_i} : \sigma \in \Sigma, \sigma(v_i) = v_i \})\]
is isomorphic to a product of centraliser subgroups in general linear groups. 

For any subset $S \subset \Mat_{n \times n}$, the centraliser $C_{\Mat_{n \times n}}(S)$ is a vector subspace of $\Mat_{n \times n}$, and thus is connected. For $S \subset \GL_n$, it follows that $C_{\GL_n}(S)$ is also connected, as this is the non-vanishing locus of a single polynomial (the determinant) in $C_{\Mat_{n \times n}}(S)$. Therefore, ${_u}\G^{\Si}$ is connected. By \cite[Proposition A.8.11]{CGP}, the fixed locus of any finite group scheme over $k$ acting on any smooth $k$-scheme is smooth; hence, ${_u}\G^{\Si}$ is smooth. Furthermore, by \cite[Proposition A.8.12]{CGP}, for any linearly reductive group scheme $H$ acting on a reductive group scheme $G$ over $k$, the connected component of the identity of the fixed locus $G^H$ with its reduced scheme structure is reductive; hence, it follows that ${_u}\G^{\Si}$ is reductive.
\end{proof}

\begin{rmk}
In fact, the above argument shows that, for any finite group $\Sigma$ acting on a product of general linear groups over a base field $k$, the fixed locus is smooth, connected and reductive (even if $k$ is not algebraically closed). However, this statement is not true for an arbitrary reductive group scheme over $k$ in positive characteristic, as it is not necessarily true that the fixed locus is connected.
\end{rmk}

We can now give an analogue of Propositions \ref{prop constr f cov}, \ref{non_empty_fibres} and \ref{prop_basic_type_map} for the $\Si$-action given by a modifying family $u$.

\begin{thm}\label{new_fSigma}
Let $u:=(u_\si)_{\si\in\Si}$ be a modifying family of elements in $\G(k)$ indexed by $\Sigma$. Then 
\begin{equation}\label{eqn Si u theta ss}
({_u}\Rep^{\Si})^{({_u}\G^{\Si},\chi_\theta)-ss}=  {_u}\Rep^{\Si}  \times_{\Rep} \Rep^{(\G,\chi_\theta)-ss}
\end{equation}
and there is a morphism
\[ {_u}f_{\Si} : {_u}\Rep^{\Si}/\!/_{\chi_\theta}\ {_u}\G^{\Si} \lra (\Mod)^{\Sigma}.\]
Furthermore, there is a moduli space of $(\Si,u,\theta)$-regularly stable representations of $Q$ given by
\[ \RepSiurs/{_u}\G^{\Si} = {_u}f_{\Si}^{-1}( \Mods), \]
where $\RepSiurs := {_u}\Rep^{\Si} \times_{\Rep} \Rep^{(\G,\chi_\theta)-s}$,
and if we denote the restriction of ${_u}f_{\Si}$ to this open subscheme by
\[ {_u}f_{\Si}^{rs} : \RepSiurs/{_u}\G^{\Si} \lra (\Mods)^{\Si}, \]
then the image of ${_u}f_{\Si}^{rs}$ is a closed subscheme of $(\Mods)^{\Si}$. 
The non-empty closed fibres of ${_u}f_{\Si}^{rs}$ are in bijection with the pointed set $$\ker\big(H^1(\Si,\Delta(k)) \lra H^1_{u}(\Si,\G(k))\big)$$ where $\Si$ acts on $\G$ via the action $\Psi^u$ defined in Lemma \ref{new_Sigma_action}. Moreover, we have  $ \im {_u}f_{\Si}^{rs}(k) \subset \cT^{-1}([c_u])$, where $[c_u]$ is the cohomology class of by the $\Delta(k)$-valued $2$-cocycle $c_u(\si_1,\si_2):=u_{\si_1} \Psi_{\si_1}(u_{\si_2}) u_{\si_1\si_2}^{-1}$ on $\Si$ defined by $u$. 
\end{thm}
\begin{proof}
The result follows from simple modifications of the arguments given in the proofs of Propositions \ref{prop constr f cov}, \ref{non_empty_fibres} and \ref{prop_basic_type_map}.
\end{proof}

\subsection{A decomposition of the fixed locus}

In this section, our goal is to give a description of the fixed locus $(\Mods)^{\Si}$ for the action of $\Si \subset \Aut^+(Q)$ in terms of the images of morphisms ${_u}f_{\Si}^{rs}$ defined by modifying families $u$ (\textit{cf.}\ Theorem \ref{new_fSigma}). We will do this in two stages: first by describing a given fibre of the type map $\cT$ as a disjoint union of images of such morphisms, and then by taking the union over all fibres of $\cT$ in order to produce a decomposition of the fixed locus  $(\Mods)^{\Si}$. Finally, we will illustrate this decomposition with some simple examples. Before giving the decomposition result for a fibre of the type map, we need a few lemmas. 

\begin{lemma}\label{equal cocycles}
Let $u'$ be a modifying family of elements of $\G(k)$ such that $c_{u'}$ 
is cohomologous to a $\Delta(k)$-valued 2-cocycle $c$, then there is a family 
$(a'_\si)_{\si\in\Si}$ of elements of $\Delta(k)$ such that $u''_\si:=a'_\si u'_\si$ is 
a modifying family of elements in $\G(k)$ with
\begin{enumerate}
\item $\Phi^{u''}_\si=\Phi^{u'}_\si$ and $\Psi^{u''}_\si = \Psi^{u'}_\si$ for all $\si\in\Si$;
\item $c_{u''} = c$ as $\Delta$-valued 2-cocycles.
\end{enumerate}
\end{lemma}
\begin{proof}
Since $c_{u'}$ and $c$ are cohomologous, there exists a family $(a'_\si)_{\si\in\Si}$ 
of elements of $\Delta(k)$ such that 
\begin{equation}\label{eq giving a}
 c(\si_1,\si_2)c_{u'}(\si_1,\si_2)^{-1} = a'_{\si_1} \si_1(a'_{\si_2}) (a'_{\si_1\si_2})^{-1} 
\quad \text{for } \si \in \Si, i =1,2.
\end{equation}
Let $u''_\si:=a'_\si u'_{\si}$, then $\Phi^{u''}_\si=\Phi^{u'}_\si$ and $\Psi^{u''}_\si= \Psi^{u'}_{\si}$, because $a'_\si\in\Delta(k)$. Since $\Delta$ is central and fixed by the $\Si$-action, it follows from \eqref{eq giving a} that $c_{u''} = c$.
\end{proof}

For a modifying family $u$, let $Z^1_{u}(\Si,\G(k))$ denote the set of $\G(k)$-valued normalised 1-cocycles on $\Si$, calculated with respect to the $\Si$-action given by $\Psi^u$. 
The following lemma describes which modifying families give the same $\Delta(k)$-valued 2-cocycle.

\begin{lemma}\label{lemma mod family 1cocycle}
Let $u = (u_\si)_{\si \in \Si}$ be a modifying family of elements in $\G(k)$ indexed by $\Si$. Then there is a bijection 
\[ \begin{array}{ccc} 
\{\text{modifying families } u'\ |\ c_{u'}=c_u \} & \longleftrightarrow & Z^1_{u}(\Si,\G(k))\\
 u'=(u'_\si)_{\si \in \Si} & \longmapsto & (b_{u'}: \Si \lra \G(k),\:\si  \longmapsto  u'_\si u^{-1}_\si)\\
 u^b= (u^b_\si:= b(\si)u_\si)_{\si \in \Si} & \longleftarrow\!\shortmid & (b: \Si \lra \G(k), \: \si \longmapsto b(\si)).
 \end{array}
\]
\end{lemma}

\begin{proof}
For a modifying family $u'$ with $c_u = c_{u'}$, we check that $b_{u'}$ is a 1-cocycle:
\begin{align*}
b_{u'}({\si_1\si_2}) & =  u'_{\si_1\si_2} u_{\si_1\si_2}^{-1} = \big(c_{u'}(\si_1,\si_2)^{-1} u'_{\si_1}\si_1(u'_{\si_2})\big) \big(\si_1(u_{\si_2}^{-1}) u_{\si_1}^{-1} c_{u}(\si_1,\si_2)\big)\\
& =  (u'_{\si_1} u_{\si_1}^{-1}) (u_{\si_1}\si_1(u'_{\si_2}u_{\si_2}^{-1})u_{\si_1}^{-1}) =  b_{u'}({\si_1}) \Psi^u_{\si_1}(b_{u'}({\si_2})),
\end{align*}
as $c_u = c_{u'}$ is valued in the central subgroup $\Delta(k) \subset \G(k)$.

For $b \in Z^1_{u}(\Si,\G(k))$, we check that $u^b$ is a modifying family with $c_{u^b} = c_u$:
\begin{eqnarray*}c_{u^b}(\si_1,\si_2) & := & u^b_{\si_1}\si_1(u^b_{\si_2}) (u^b_{\si_1\si_2})^{-1}  =  b({\si_1}) u_{\si_1} \si_1(b({\si_2})u_{\si_2}) (b({\si_1\si_2})u_{\si_1\si_2})^{-1}\\
& \: = & b({\si_1}) (\underbrace{u_{\si_1}\si_1(b({\si_2}))u_{\si_1}^{-1}}_{=\Psi^u_{\si_1}(b({\si_2}))}) \underbrace{u_{\si_1}\si_1(u_{\si_2}) u_{\si_1\si_2}^{-1}}_{=c_u(\si_1,\si_2)\in\Delta} b({\si_1\si_2})^{-1}\\
& \: = & c_u(\si_1,\si_2).
\end{eqnarray*}
This completes the proof, as clearly these two maps are inverse to each other.
\end{proof}

\begin{rmk}
Note that two $\G$-valued $1$-cocycles $b_{u'}$ and $b_{u''}$ are cohomologous (that is, there exists $g \in \G(k)$ such that, $b_{u''}(\si)=gb_{u'}(\si)\Psi^u_\si(g^{-1})$ for all $\si\in \Si$) if and only if  there exists $g \in \G(k)$ such that $u''_\si = g u'_\si \si(g^{-1})$ for all $\si\in \Si$. 
\end{rmk}

For a modifying family $u$, we can now give a decomposition of the fibre of the type map $\cT : (\Mods)^{\Si}(k) \lra H^2(\Si,\Delta(k))$ over $[c_u]$, where the indexing set is the orbit space $H^1_{u}(\Si,\G(k)) / H^1(\Si,\Delta(k))$: as $\Delta(k)$ is central and $\Si$-invariant in $\G(k)$, the map $(\beta,b)\lmt (\beta b)_\si :=(\beta_\si b_\si)_{\si\in\Si}$ indeed induces an action of the group $H^1_{u}(\Si,\Delta(k))$ on the set $H^1_{u}(\Si,\G(k))$. Moreover, $H^1_{u}(\Si,\Delta(k))$ is independent of the modifying family $u$, as the $\Si$-action on $\Delta$ induced by $\Psi^u$ is trivial. For $[b]\in H^1_{u}(\Si,\G(k))$, we shall denote its $H^1(\Si,\Delta(k))$-orbit by $\ov{[b]}$. 

Let us first show that the image of the morphism ${_{u^{b}}}f_{\Si}^{rs}$ (\textit{cf.}\ Theorem \ref{new_fSigma}) is independent of the choice of 1-cocycle $b \in \ov{[b]}$.

\begin{lemma}\label{lemma cohomologous implies same images}
For $b_1,b_2 \in Z^1_{u}(\Si,\G(k))$, let $u^{b_1}, u^{b_2}$ be the associated modifying families given by Lemma \ref{lemma mod family 1cocycle}. If $\ov{[b_1]} = \ov{[b_2]}$, then $ \im {_{u^{b_1}}}f_{\Si}^{rs} = \im {_{u^{b_2}}}f_{\Si}^{rs}$.
\end{lemma}
\begin{proof}
Let us first show this statement when $[b_1] = [b_2]$; that is $b_1$ and $b_2$ are cohomologous, so there exists $g \in \G(k)$ such that $b_2(\si)=gb_{1}(\si)\Psi^u_\si(g^{-1})$ for all $\si\in \Si$. Then the action of $g$ gives an isomorphism ${_{u^{b_1}}}\Rep^{\Si} \overset{\simeq}{\lra} {_{u^{b_2}}}\Rep^{\Si}$ and the claim follows, as $\pi$ is $\G$-invariant. 

If $b_1 = \beta \cdot b_2$ for a $\Delta(k)$-valued 1-cocycle $\beta$, then the associated modifying families $u^{b_1}, u^{b_2}$ define the same actions on $\Rep$ and $\G$, as $\beta_{\sigma} \in \Delta$ and this group acts trivially on $\Rep$ and is central in $\G$. This completes the proof.
\end{proof}

We thus let $ \im {_{u^{\ov{[b]}}}}f_{\Si}^{rs}:=  \im {_{u^{b}}}f_{\Si}^{rs}$ for any representative $b$ of $\ov{[b]}$.

\begin{thm}\label{components_of_fibres_of_the_type_map}
Let $u$ be a modifying family of elements in $\G(k)$ indexed by $\Si$; then there is a decomposition of $\cT^{-1}([c_u])$ into a disjoint union of closed subvarieties
\[ \cT^{-1}([c_u]) = \bigsqcup_{\ov{[b]} \in H^1_{u}(\Si,\G(k)) / H^1(\Si,\Delta(k))} (\im {_{u^{\ov{[b]}}}}f^{rs}_{\Si})(k).\] 
Furthermore, the non-empty fibres of ${_{u^b}}f_{\Si}^{rs}$ are described by Theorem \ref{new_fSigma}.
\end{thm}
\begin{proof} 
To show that these images cover $\cT^{-1}([c_u])$, let $\G\cdot M\in \cT^{-1}([c_u])$; then by definition of the type map (\textit{cf.}\ Construction \ref{constr_type_map_def}), there exists a modifying family $u'=(u'_\si)_{\si\in\Si}$ of elements in $\G(k)$ such that $M \in {_{u'}}\Rep^{\Si}$ and $[c_{u'}]=[c_u]$. 
By Lemma \ref{equal cocycles}, we can assume that $c_{u'}=c_u$ and so, by Lemma \ref{lemma mod family 1cocycle}, there exists $b \in Z^1_{u}(\Si,\G(k))$ such that $u' = u^b$. Hence, $\G \cdot M \in \im {_{u^b}}f_{\Si}^{rs}$.

To prove that this union is disjoint, suppose that $\G \cdot M \in \im {_{u^{b_i}}}f^{rs}_{\Si}$ for $ i =1,2$. Then there exist modifying families $u_i$, for $i=1,2$, such that 
\begin{enumerate}
\item $[b_{u_i}]=[b_i] \in H^1_{u}(\Si,\G(k))$,
\item $M \in {_{u_i}}\Rep^{\Si}$ for $i=1,2$,
\item $c_{u_1} = c_{u_2} =c_u$.
\end{enumerate}
The only one of these assertions which is not clear is the final one, which follows from the fact that $c_{u^{b_i}}=c_u$, for $i=1,2$, and the observation that if $[b] = [b']$, then $c_{u^b} =c_{u^{b'}}$. From (2), we deduce that $a_{\si}:=u_{2,\si} u_{1,\si}^{-1} \in \Stab_{\G}(M) = \Delta(k)$, from which we conclude that $b_{u_2,\si} = a_{\si} b_{u_1,\si}$ for all $\si$, therefore that $[b_{u_1}]$ and $[b_{u_2}]$ lie in the same $H^1(\Si,\Delta(k))$-orbit in $H^1_{u}(\Si,\G(k))$. This completes the proof.
\end{proof}

By definition of the type map $ \cT : (\Mods)^{\Si}(k) \lra  H^2(\Si,\Delta(k))$, if $[c] \in \im \cT$, there exists a modifying family $u$ with $[c] = [c_u]$. We can now prove Theorem \ref{decomp_thm_qaut_intro}.

\begin{proof}[Proof of Theorem \ref{decomp_thm_qaut_intro}]\label{proof_main_result}
We have a set-theoretic decomposition of the closed points 
\[ (\Mods)^{\Si}(k) = \bigsqcup_{[c] \in \im \cT} \cT^{-1}([c])\]
and, for each $[c] \in \im \cT$, there exists a modifying family $u$ such that $[c] = [c_u]$. Hence, by Theorem \ref{components_of_fibres_of_the_type_map} we obtain the claimed decomposition on the level of $k$-points. The fibres of the morphisms ${_{u^b}}f^{rs}_{\Si}$ are described as in Theorem \ref{new_fSigma}. 

By Proposition \ref{prop image frs closed}, $\im {_{u^b}}f^{rs}_{\Si}$ is a closed subvariety of $(\Mods)^{\Si}$. Thus, as we are working with varieties over an algebraically closed field $k$, this set-theoretic decomposition on the level of closed points arises from a decomposition of varieties if the index set of the decomposition is finite. Finally, the finiteness of the index set in the two cases stated in this theorem follows from Corollary \ref{cor index finite}.
\end{proof}

\begin{rmk}\label{rmks on decomp thm}We make the following observations on Theorem \ref{decomp_thm_qaut_intro}.
\begin{enumerate}
\item $\Mods$ is smooth by Luna's \'{e}tale slice Theorem, as it is a geometric quotient of the smooth $k$-variety $\Rep^{\chi_\theta-s}$ by the free action of the reductive group $\Gbar:=\G/\Delta$. Hence, $(\Mods)^{\Si}$ is smooth by  \cite[Proposition A.8.11]{CGP}.
\item Over the complex numbers $\Mods$ is a smooth K\"{a}hler manifold, as it is homeomorphic, by the Kempf-Ness Theorem \cite{kempf_ness}, to the K\"{a}hler reduction of the complex vector space $\Rep$ (by the action of the maximal compact subgroup of $\G$). The group $\Si$ acts algebraically on $\Mods$, therefore preserves the complex structure. In particular, $(\Mods)^{\Sigma}$ is a holomorphic, and thus K\"ahler, submanifold of $\Mods$. 
\item\label{non-closed field} If $k$ is not algebraically closed, then one has to replace $\theta$-stability with a stronger notion of $\theta$-geometric stability, which is stable under base change and we instead consider the moduli space $\Modgs$ of $\theta$-geometrically stable representations (\textit{cf.}\ \cite[$\S$2]{HS_galois}). In this case, we obtain a set-theoretic decomposition of $(\Modgs)^{\Si}(\Omega)$ for any algebraically closed field $\Omega/k$ via the same techniques. More precisely, one can construct the morphism $f_{\Sigma}$ using the universal property of the GIT quotient and noting that it suffices to check the first equality in Proposition \ref{prop constr f cov} on $\kb$-points, which we have already done. Then one can prove an analogous description to Proposition \ref{non_empty_fibres} for all geometric fibres of $f^{rs}_\Sigma$ and, for any algebraically closed field $\Omega/k$, one can construct a type map 
\[ \cT_{\Omega} : (\Modgs)^{\Si}(\Omega) \lra H^2(\Si, \Delta(\Omega)) \]
as in Construction \ref{constr_type_map_def}, by using Remark \ref{stabiliser_of_stable} for $\Omega$. 
\end{enumerate}
\end{rmk}

We end this section with two examples that illustrate the decomposition theorem.

\begin{ex}\label{ex f not surj cov}
Let $Q$ be the quiver
$ \xymatrix@1{  1 \bullet \ar@/^/[r]^{a} \ar@/_/[r]_{b}  & \bullet 2  } $
with covariant involution $\sigma$ which fixes the vertex set $V = \{ 1,2\}$ 
and sends arrow $a$ to $b$. As all dimension vectors $d$ and stability parameters 
$\theta$ are $\sigma$-compatible, we choose $d = (1,1)$ and $\theta = (1,-1)$. 
Then $(s_1,s_2) \in \G = \GG_m^2$ acts on $M= (M_a,M_b) \in \Rep = \AA^2$ by
\[ (s_1,s_2) \cdot (M_a,M_b) = (s_2 M_a s_1^{-1},s_2 M_b s_1^{-1}).\] Let us write $\AA^2 = \spec k[X_a,X_b]$; then both $X_a$ and $X_b$ are semi-invariant functions 
for the character $\chi_\theta$ defined at \eqref{the_character}. Hence, we have the GIT quotient
\[ \pi: \Rep^{\theta-ss} = \AA^2 - \{ 0 \} \lra 
\Mod :=\Rep/\!/_{\chi_\theta} \G = \proj k[X_a,X_b] = \PP^1.\]
In fact, the $\theta$-semistable locus and $\theta$-stable locus agree, 
so that this is a geometric quotient, and, moreover, every GIT semistable orbit has stabiliser group $\Delta(k)$.  

\noindent The involution $\sigma$ on $Q$ induces an involution $\Phi_\sigma$ on $\Rep$ given by
\[ \Phi_{\sigma}((M_a,M_b)) = (M_b,M_a)\]
and induces a trivial action on $\G$, as $\sigma$ fixes the vertex set. 
The induced involution $\Phi'_\sigma$ on $\Mod=\PP^1$ is given by
$ \Phi'_\sigma([M_a:M_b]) = [M_b:M_a]$
and the fixed locus is $(\PP^1)^{\sigma}=\{[1:1],[1:-1]\}$. The GIT quotient of $\G^\sigma = \G$ acting on $\Rep^\sigma = \{(M_a,M_b): M_a = M_b\} \cong \AA^1$ with respect to $\chi_\theta$ is
\[ (\Rep^{\sigma})^{\theta-ss} \cong \AA^1 - \{ 0 \} \lra 
\Rep^\sigma/\!/_{\chi_\theta} \G = \proj k[X_a] = \spec k, \]
which we can interpret this quotient as a moduli space for $\theta$-semistable 
representations of dimension $\tilde{d}=(1,1)$ of the quotient quiver $Q/\sigma = \bullet \lra \bullet $ 
by Corollary \ref{modular_interp_of_Si_fixed_reps}. The morphism 
\[f_\sigma : \Rep^\sigma/\!/_{\chi_\theta} \G \cong \spec k \lra 
(\Mod)^{\sigma}=\{[1:1],[1:-1]\}\]
is injective but, when $\mathrm{char}(k)\neq 2$, it is not surjective, as its image equal to the point $[1:1]$. Let $u$ be the modifying family given by $u_\si = (1,-1) \in \G(k)$; then we can modify the $\ZZ/2\ZZ$-action on $\Rep$ by 
\[ \Phi^u_\sigma((M_a,M_b)) := u_\si \cdot \Phi_\sigma((M_a,M_b)) = (-M_b,-M_a) \]
and the modified action on $\G$ remains trivial. Then the GIT quotient of ${_u}\G^{\si} = \G$ acting on ${_u}\Rep^{\si} = \{(M_a,M_b): M_a = - M_b\} \cong \AA^1$ with respect to $\chi_\theta$ is a point and the image of 
${_u}f_{\si}: {_u}\Rep^{\si}/\!/_{\chi_\theta}\ {_u}\G^{\si} \lra (\Mod)^{\sigma}$
is the point $[1:-1] \in (\PP^1)^\sigma$. Hence, the fixed locus decomposes into two pieces 
\begin{equation}\label{example_of_fibre_of_the_type_map}
(\Mod)^\sigma =  \{[1:1],[1:-1]\} \cong \Rep^\sigma/\!/_{\chi_\theta} \G^\sigma \bigsqcup {_u}\Rep^{\si}/\!/_{\chi_\theta}\ {_u}\G^{\si}.
\end{equation}

Let us now explain this decomposition in terms of the group cohomology of $\Si$. First, the injectivity of $f_\si$ follows from Proposition \ref{non_empty_fibres} and the fact that $\ZZ/2\ZZ$ acts trivially on $\G$, so that $H^1(\Si,\Delta(k)) \simeq \{a\in k\ |\ a^2=1_k\} = \{\pm 1_k\}$ and 
$$H^1(\Si,\G(k)) \simeq \{(s_1,s_2)\in\GG_m(k)\times\GG_m(k)\ |\ s_1^2=s_2^2=1_k\} \simeq \{\pm 1_k\}\times\{\pm 1_k\}$$ and the map $H^1(\Si,\Delta(k))\lra H^1(\Si,\G(k))$ is conjugate to the group homomorphism $a\lmt (a,a)$, which is injective. There is only one fibre for the type map, because $H^2(\Si,\Delta(k))=1$ by Example \ref{ex H2 cyclic gp}. This fibre has two components, as 
\begin{align*} H^1(\Si,\G(k)) / H^1(\Si,\Delta(k)) & =   \big(\{\pm 1_k\}\times\{\pm 1_k\}\big) / \{\pm 1\}\\ & =  \{(1,1), (1,-1)\},
\end{align*} 
which has two distinct elements if $\mathrm{char}(k)\neq 2$; this gives the decomposition \eqref{example_of_fibre_of_the_type_map}. 
\end{ex}

\begin{ex}\label{ex f not inj cov}
Let $Q$ be the quiver 
$ \xymatrix@1{  1 \bullet \ar@/^/[r]^{a} & \bullet 2 \ar@/^/[l]^{b}   } $
with covariant involution $\sigma$ which sends vertex $1$ to $2$, and sends arrow $a$ to $b$. 
Then a dimension vector $d=(d_1,d_2)$ is $\sigma$-compatible if and only if $d_1 = d_2$ and 
similarly for $\theta$. For $\sum_i \theta_i d_i = 0$, we need $\theta = (0,0)$, if $\theta$ and $d$ are both $\sigma$-compatible. 
Let  $d = (1,1)$ and $\theta = (0,0)$; 
then $(s_1,s_2) \in \G = \GG_m^2$ acts on $(M_a,M_b) \in \Rep = \AA^2$ by
\[(s_1,s_2) \cdot (M_a,M_b) = (s_2 M_a s_1^{-1},s_1 M_b s_2^{-1}) \]
and the affine GIT quotient of this action is
\[ \pi : \Rep \lra \Mod:=\Rep/\!/\G = \spec k[X_aX_b] \cong \AA^1;\]
this restricts to a geometric quotient on the stable locus, which is the complement to 
the union of the coordinate axes. The action of $\sigma$ on $\G$ is given by 
$(s_1,s_2) \longmapsto (s_2,s_1)$ and so $\G^\sigma = \Delta $, and the $\si$-action on 
$\Rep$ is given by $(M_a,M_b) \longmapsto (M_b,M_a)$. Hence, the induced 
action of $\sigma$ on $\Mod \cong \AA^1$ is trivial. 

The action of $\G^\sigma = \Delta$ on $\Rep^{\sigma}\cong\AA^1$ is 
also trivial and so the affine GIT quotient of this action is the identity map on $\AA^1$. 
We can interpret this as a moduli space for the quotient quiver $Q/\si$, which is the Jordan quiver with one vertex and a single loop.

Therefore, the morphism
\[ f_\sigma: \Rep^{\sigma}/\!/\G^\sigma \cong \AA^1/\!/\GG_m = \AA^1 
\lra (\Mod)^\sigma = (\AA^1)^\sigma = \AA^1\]
is given by $z \longmapsto z^2$, which is not injective when $\mathrm{char}(k)\neq 2$. 
More concretely, we can see this on the level of representations. 
The two representations $M^{\pm} = (\pm1,\pm1) \in \Rep^\sigma$ correspond 
to the same stable $\G$-orbit (for example, if $g = (-1,1) \in \G$, then 
$g \cdot M^{+} = M^{-}$). However, the $\G^\sigma$-orbits of these points 
are distinct, as the $\G^\sigma$-action is trivial.

To see the failure of injectivity on the level of group cohomology, we use Proposition \ref{non_empty_fibres} and the fact that $\ZZ/2\ZZ$ acts on $\G = \GG_m^2$ by swapping the two factors, so that $H^1(\Si,\Delta(k))\simeq\{\pm 1_k\}$ and $$H^1(\Si,\G(k)) \simeq \frac{ \{(s_1,s_2)\in \GG_m(k)\times \GG_m(k)\ |\ (s_1 s_2,s_2s_1)=(1,1)\} }{ \{(\frac{s_1}{s_2},\frac{s_2}{s_1}) : (s_1,s_2)\in \GG_m(k) \times \GG_m(k)\}} \simeq \{(1,1)\},$$ which shows that $f_\Si^{rs}$ is $2:1$ when $\mathrm{char}(k)\neq 2$.

Finally $H^2(\Si,\Delta(k) ) = {1}$ by Example \ref{ex H2 cyclic gp}, and $H^1(\Si,\G(k)) = \{(1,1)\}$ so, by Theorem \ref{components_of_fibres_of_the_type_map}, one has $(\Mods)^{\si}=\cT^{-1}([1]) = \mathrm{Im}\, f_{\Si}^{rs}$, as we saw above.
\end{ex}

\subsection{Representation-theoretic interpretration}\label{representation_theoretic_interp}

Let $\crep_k(Q)$ denote the category of $k$-representations of $Q$.

\begin{defn}
A covariant (resp.\ contravariant) automorphism $\sigma$ of $Q$ determines a covariant (resp.\ contravariant) functor $\sigma: \crep_k(Q) \lra \crep_k(Q)$, which on a $k$-representation $W=((W_v)_{v \in V}, (\varphi_a)_{a \in A})$ is given by
\[ \sigma(W) = \left\{ \begin{array}{ll} ((W_{\sigma(v)})_{v \in V}, (\varphi_{\sigma(a)})_{a \in A})& \text{if } \sigma \text{ is covariant,} \\ ((W_{\sigma(v)}^*)_{v \in V}, (\varphi_{\sigma(a)}^*)_{a \in A}) & \text{if } \sigma \text{ is contravariant.} \end{array} \right. \]
\end{defn}

\noindent Note that if $W$ has dimension $d=(d_v)_{v\in V}$ then $\si(W)$ has dimension $\si(d)=(d_{\si(v)})_{v\in V}$. In order for subrepresentations of $\si(W)$ to correspond canonically and bijectively to subrepresentations of $W$, we restrict ourselves to the covariant case $\Si \subset \Aut^+(Q)$ in this subsection.

\begin{lemma}\label{theta_semistability_preserved}
Let $\theta=(\theta_v)_{v\in V}$ be a $\Si$-compatible stability parameter. Then, for all $\si\in\Si$, a $k$-representation $W$ of $Q$ is $\theta$-(semi)stable if and only if $\sigma(W)$ is $\theta$-(semi)stable. 
\end{lemma}

\begin{proof}
We note that for all $k$-representations $W$, we have, by $\Sigma$-compatibility of $\theta$,
\[ \mu_{\theta}(\sigma(W))= \frac{\sum_{v \in V} \theta_v \dim W_{\sigma(v)}}{\sum_{v \in V} \dim W_{\sigma(v)}} = \frac{\sum_{v \in V} \theta_{\sigma^{-1}(v)} \dim W_v}{\sum_{ v\in V} \dim W_v} = \mu_{\theta}(W).\] As $W' \subset W$ is a subrepresentation if and only if $\sigma(W') \subset \sigma(W)$ is a subrepresentation, this equality proves the statement.
\end{proof}

\begin{defn}[Equivariant representations]
For $\Si \subset \Aut^+(Q)$, let $\crep_k(Q,\Si)$ denote the category whose objects are pairs $(W,\ga)$ consisting of an object $W$ of $\crep_k(Q)$ and a family of isomorphisms $(\ga_\si:\si(W) \overset{\simeq}{\lra} W)_{\si\in \Si}$ such that, for all $(\si_1,\si_2)\in\Si\times\Si$, we have $\ga_{\si_1\si_2} = \ga_{\si_1}\si_1(\ga_{\si_2})$, and whose morphisms are morphisms of representations that commute to the $(\ga_\si)_{\si\in\Si}$.
\end{defn}

\noindent There is a faithful forgetful functor $\crep_k(Q,\Si) \lra \crep_k(Q)$, and a representation of $Q$ can only lie in the essential image of this functor, if its dimension vector is $\Si$-compatible. More generally, given a $2$-cocycle $c\in Z^2(\Si,k^\times)$, we define $(\Si,c)$-equivariant representations of $Q$ to be pairs $(W,\ga)$ as above, except that we now ask for  $\ga_{\si_1}\si_1(\ga_{\si_2}) = c(\si_1,\si_2) \ga_{\si_1\si_2}$. This defines a category $\crep_k(Q,\Si,c)$ analogously to the case of $c \equiv 1$. The following result is then easily checked.

\begin{lemma}
If $u=(u_\si)_{\si\in\Si}$ is a modifying family in the sense of Definition \ref{modifying_family_def} and $c_u$ is the associated $2$-cocycle, there is a bijection between isomorphism classes of $d$-dimensional objects of $\crep_k(Q,\Si,c_u)$ and the set ${_u}\Rep^{\Si}(k) /\ {_u}\G^{\Si}(k)$.
\end{lemma}

\noindent The point of the above result is that there is a natural notion of $\theta$-semistability in $\crep_k(Q,\Si,c_u)$, that will eventually coincide with GIT semistability for the ${_u}\G^{\Si}$-action on ${_u}\Rep^{\Si}$ (with respect to the character $\chi_\theta|_{{_u}\G^{\Si}}$) for $\Si$-compatible $\theta$; see Theorem \ref{comparison_with_GIT_stability}.

\begin{defn}[$(\Si,\theta)$-(semi)stability]
Let $c\in Z^2(\Si,k^\times)$. A $(\Si,c)$-equivariant representation $(W,\ga)$ of $Q$ is called $(\Si,\theta)$-(semi)stable if, for all non-zero proper subrepresentations $(W',\ga')$ in $\crep(Q,\Si,c)$, one has $\mu_\theta(W') (\leq) \mu_\theta(W)$.
\end{defn}

We henceforth fix a $\Si$-compatible stability parameter $\theta$. Let us provide a representation-theoretic version of Lemma \ref{lemma compare ss} (see Theorem \ref{comparison_with_GIT_stability}).

\begin{prop}\label{sst_as_an_equiv_rep}
Let $(W,\ga)$ be a $(\Si,c)$-equivariant representation of $Q$. Then the following are equivalent:
\begin{enumerate}
\item $W$ is $\theta$-semistable as a representation of $Q$.
\item $(W,\ga)$ is $(\Si,\theta)$-semistable as a $(\Si,c)$-equivariant representation of $Q$.
\end{enumerate}
\end{prop}

\begin{proof}
Evidently, (1) implies (2). For the converse, we proceed by contrapositive using the uniqueness of the strictly contradicting semistability subrepresentation (\textit{scss}) $U\subset W$ (\textit{cf.}\ \cite[Definition 2.3]{HS_galois}). Let us assume that $W$ is not $\theta$-semistable. Since subrepresentations of $\si(W)$ correspond bijectively to subrepresentations of $W$ and $\theta$ is $\Si$-compatible, we have that $\si(U)$ is the \textit{scss} of $\si(W)$ for each $\si\in \Si$. As $\ga_\si:\si(W)\lra W$ is an isomorphism, $\ga_\si(\si(U))$ is the \textit{scss} of $W$, thus \ $\ga_\si(\si(U)) = U$. In particular, $(U,\ga|_U)$ is a $(\Si,c)$-equivariant subrepresentation of $(W,\ga)$, which can therefore not be $(\Si,\theta)$-semistable.
\end{proof}

\noindent For $(W,\ga)\in \crep_k(Q,\Si,c)$, one can construct a unique Harder--Narasimhan filtration of $(W,\ga)$ with respect to $(\Si,\theta)$-semistability. Proposition \ref{sst_as_an_equiv_rep} then has the following corollary.

\begin{cor}
Let $(W,\ga)$ be a $(\Si,c)$-equivariant representation of $Q$. Then the Harder--Narasimhan filtration of $(W,\ga)$ with respect to $(\Si,\theta)$-semistability agrees with the Harder--Narasimhan filtration of $W$ with respect to $\theta$-semistability.
\end{cor}

\noindent 
If we replace semistability by stability, the statement of Proposition \ref{sst_as_an_equiv_rep} is no longer true in general, which can be remedied by introducing the following notion.

\begin{defn}[$(\Si,\theta)$-regularly stable]\label{rep_theorectic_rs}
A $(\Si,c)$-equivariant $k$-representation $(W,\ga)$ is called $(\Si,\theta)$-regularly stable if $W$ is $\theta$-stable as a $k$-representation.
\end{defn}

\begin{thm}\label{comparison_with_GIT_stability}
Let $u$ be a modifying family and let $c_u$ be the associated $2$-cocycle. Let $M\in {_u}\Rep^{\Si}$ and let $(W,\ga)$ be the corresponding $(\Si,c_u)$-equivariant representation of $Q$. Then:
\begin{enumerate}
\item $M$ is $({_u}\G^{\Si},\chi_\theta)$-semistable in the GIT sense if and only if $(W,\ga)$ is $(\Si,\theta)$-semistable as a $(\Si,c_u)$-equivariant representation.
\item $M$ is $(\Si,u,\chi_\theta)$-regularly stable in the sense of Definition \ref{GIT_rs} if and only if $(W,\ga)$ is $(\Si,\theta)$-regularly stable in the sense of Definition \ref{rep_theorectic_rs}.
\end{enumerate}
\end{thm}

\begin{proof}
We will prove (1) using the Hilbert--Mumford criterion; then (2) follows along the same lines, in view of Definitions \ref{GIT_rs} and \ref{rep_theorectic_rs}.

Given splittings $W_v :=k^{d_v}= W_v' \oplus W_v''$ for all $v \in V$ and integers $r' > r''$, we can define a 1-parameter subgroup (1-PS) $\lambda$ of $\G$ by
\begin{equation}\label{1PS_subreps}
 \lambda(t) =\left\{  \left( \begin{array}{cc} t^{r'}\text{Id}_{W_v'} & 0 \\ 0 & t^{r''}\text{Id}_{W_v''} \end{array} \right) \right\}_{v \in V}
\end{equation}
and we assume that $\lim_{t \lra 0} \lambda(t) \cdot M$ exists, which is equivalent to the statement that the subspaces $W_v'$ determine a subrepresentation $M'$ of $M$ by \cite[$\S$3]{king}. Moreover, the inequality $\mu_\theta(M') \leq \mu_\theta(M)$ is equivalent to the inequality $\langle \chi_\theta,\lambda \rangle \geq 0$, where $\langle -,- \rangle$ denotes the natural pairing between characters and co-characters of $\G$. If we view $u_{\sigma,v} \in \GL_{d_v}(k)$ as an isomorphism $u_{\sigma,v} : W_{\sigma(v)} \lra W_v$, then we see that the above 1-PS $\lambda$ factors through $_u \G^\Si$ if and only if $u_{\sigma,v}(W_{\sigma(v)}') = W_v'$ for all $\sigma \in \Si$ and $v \in V$; that is, if and only if $M'$ is a $(\Si,c_u)$-equivariant subrepresentation, where the equivariant structure is given by the isomorphisms $u_\sigma$. By the Hilbert--Mumford criterion, $M$ is $({_u}\G^{\Si},\chi_\theta)$-semistable if and only if for all 1-PSs $\lambda$ of  $_u \G^\Si$ for which $\lim_{t \lra 0} \lambda(t) \cdot M$ exists, we have $\langle \chi_\theta,\lambda \rangle \geq 0$, and by induction on the length of the filtration it suffices to assume our 1-PSs are of the form in \eqref{1PS_subreps}. By the above equivalences, the $({_u}\G^{\Si},\chi_\theta)$-semistability of $M$ is equivalent to $M$ being $(\Si,\theta)$-semistable as a $(\Si,c_u)$-equivariant representation.
\end{proof}

\noindent We thus obtain a modular interpretation of the GIT quotients forming the domains of the morphisms ${_u}f_{\Si}$ and ${_u}f_{\Si}^{rs}$ introduced in Theorem \ref{new_fSigma}. The proof is similar to the proof of Theorem \ref{GIT_const_of_Mod}.

\begin{cor}\label{modular_interp_of_Sigma_fixed_pts_as_Sigma_equiv_rep}
The GIT quotient 
$$\mathcal{M}_{Q,d}^{(\Si,u,\theta)-ss} := {_u}\Rep^{\Si} /\!/_{\chi_\theta} \ {_u}\G^{\Si} $$ $$(\mathrm{resp.}\ \ModSiurs := ({_u}\Rep^{\Si})^{\chi_\theta-rs} / \ {_u}\G^{\Si})$$
is a coarse moduli space for $(\Si,\theta)$-semistable (resp.\ $(\Si,\theta)$-regularly stable) $(\Si,c_u)$-equivariant $d$-dimensional $k$-representations of $Q$.
\end{cor}

\noindent The points of $\ModSiurs$ are isomorphism classes of $(\Si,\theta)$-regularly stable $(\Si,c_u)$-equivariant $d$-dimensional $k$-representations of $Q$ and, as in the proof of Theorem \ref{GIT_const_of_Mod}, we can interpret the points of $\mathcal{M}_{Q,d}^{(\Si,u,\theta)-ss}$ as $S$-equivalence classes of $(\Si,\theta)$-semistable $(\Si,c_u)$-equivariant $d$-dimensional representations of $Q$ by noting that every $(\Sigma,\theta)$-semistable $(\Si,c)$-equivariant representation has a Jordan-H\"older filtration by $(\Si,c)$-equivariant subrepresentations whose successive quotients are $(\Si,\theta)$-stable; this Jordan--H\"older filtration is not necessarily 
unique but the associated graded object is and we say that two $(\Si,\theta)$-semistable $(\Si,c_u)$-equivariant representations are $S$-equivalent if the associated graded objects 
for their respective Jordan--H\"older filtrations are isomorphic as $(\Si,c_u)$-equivariant representations; the desired modular interpretation of points of the GIT quotient ${_u}\Rep^{\Si} /\!/_{\chi_\theta}\ {_u}\G^{\Si}$ then follows by the same arguments as in Theorem \ref{GIT_const_of_Mod}, using again the results of \cite{king} to relate $S$-equivalence in the representation-theoretic sense to equivalence of semistable orbits in the GIT setting.

As a consequence, we can revisit Theorem \ref{decomp_thm_qaut_intro} as follows. For the $\Si$-action on $\Mods$, there is a decomposition
\[ (\Mods)^{\Si} = \bigsqcup_{\begin{smallmatrix}[c_u] \in \im \cT \\ \ov{[b]} \in H^1_{u}(\Si, \G(k)) / H^1(\Si,\Delta(k)) \end{smallmatrix}} {_u}f^{rs}_{\Si} \left( \mathcal{M}^{(\Si,u^b,\theta)-rs}_{Q,d}\right)\]
where $u^b$ is a modifying family determined by $[b]$ and $u$. If $H^1(\Si,\Delta(k))=1$, then 
\[(\Mods)^{\Si} \cong \bigsqcup_{\begin{smallmatrix}[c_u] \in \im \cT \\ \ov{[b]} \in H^1_{u}(\Si, \G(k)) \end{smallmatrix}} \mathcal{M}^{(\Si,u^b,\theta)-rs}_{Q,d},\]
as one can deduce that the morphisms ${_{u^b}}f_{\Si}^{rs}$ are all closed immersions, by using the fact that $H^1(\Si,\Delta(k))=1$ and Theorem \ref{new_fSigma}.

\subsection{Actions by arbitrary groups of quiver automorphisms}\label{sec cov and contr case}

We now consider a subgroup $\Sigma \subset \Aut(Q)$ that contains at least one contravariant automorphism, as otherwise $\Sigma \subset \Aut^+(Q)$ and this is studied above. By restricting the homomorphism $\text{sign}: \Aut(Q) \lra \{\pm 1\}$ to $\Sigma$, we obtain a short exact sequence
\[ 1 \lra \Sigma^+ \lra \Sigma \lra \{ \pm 1\} \lra 1,\]
where $\Sigma^+ \subset \Aut^+(Q)$.

\begin{defn}\label{Si_comp_dim_vector_and_stab_param_general}
For $\Sigma \subset \Aut(Q)$, we make the following definitions.
\begin{enumerate}
\item A dimension vector $d=(d_v)_{v\in V}$ is $\Sigma$-compatible if $\sigma(d) = d$ for all $\sigma \in \Sigma$.
\item A stability parameter $\theta=(\theta_v)_{v\in V}$ is $\Sigma$-compatible if $\sigma(\theta) = \text{sign}(\sigma)\theta$ for all $\sigma \in \Sigma$.
\end{enumerate}
\end{defn}

For a $\Si$-compatible dimension vector $d$, we can construct induced actions of $\Sigma$ on $\Rep$ and $\G$ as follows: for $\sigma \in \Sigma$, we define automorphisms
\begin{equation}\label{action_on_reps_with_sign}
 \Phi_\sigma : \Rep \lra \Rep, \quad 
(M_a)_{a \in A} \lmt \left\{\begin{array}{ll}
(M_{\sigma(a)})_{a \in A} & \text{if } \text{sign}(\sigma) = 1,\\
 (^{t}M_{\sigma(a)})_{a \in A} & \text{if } \text{sign}(\sigma) = -1, \end{array} \right.
 \end{equation}
and 
\begin{equation}\label{action_on_group_with_sign}
\Psi_\sigma : \G \lra\G, \quad (g_v)_{v \in V} \lmt \left\{\begin{array}{ll}
(g_{\sigma(v)})_{v \in V} & \text{if } \text{sign}(\sigma) = 1,\\
 (^{t}g_{\sigma(v)}^{-1})_{v \in V} & \text{if } \text{sign}(\sigma) = -1. \end{array} \right. 
 \end{equation}
These $\Sigma$-actions are compatible with the $\G$-action on $\Rep$ in the sense that
\begin{equation}\label{mixed inv action comp}
\Phi_\sigma(g \cdot M) = \Psi_\sigma(g) \cdot \Phi_\sigma(M).
\end{equation}

Let $Q/\Sigma^+=(V/\Sigma^+,A/\Sigma^+, \tilde{h}, \tilde{t})$ denote the quotient quiver (\textit{cf.}\ Definition \ref{defn quotient quiver}).

\begin{lemma}\label{lemma induced invol}
Any automorphism $\sigma \in \Sigma < \Aut(Q)$ descends to an automorphism $\tilde{\sigma} \in \Aut(Q/\Sigma^+)$, which has the same sign as $\sigma$. This defines a group homomorphisms $\Sigma \lra \Aut(Q/\Sigma^+)$. Furthermore, the following statements hold.
\begin{enumerate}
\item If $\sigma$ is covariant, then $\tilde{\sigma}$ is the identity.
\item If $\sigma$ is contravariant, then $\tilde{\sigma}^2$ is the identity.
\end{enumerate}
\end{lemma}
\begin{proof}
Since $\Sigma^+ < \Sigma$ is a normal subgroup, it follows that $\sigma$ descends to a well-defined automorphism $\tilde{\sigma}$ of $Q/\Sigma^+$ given by
\[ \tilde{\sigma}( \Sigma^+ \cdot v) := \Sigma^+ \cdot \sigma(v)   \quad \text{and} \quad 
\tilde{\sigma}( \Sigma^+ \cdot a): = \Sigma^+ \cdot \sigma(a).\]
Moreover, $\tilde{\sigma}$ has the same sign as $\sigma$ and this defines a group homomorphism. If $\sigma$ is covariant, then it is an element of the subgroup $\Sigma^+$ and so $\tilde{\sigma}$ is the identity. If $\sigma$ is contravariant, then $\sigma^2$ is covariant and so it is the identity.
\end{proof}

By assumption that $\Sigma^+ \neq \Sigma$, there is a contravariant automorphism $\sigma \in \Sigma$ and this induces a contravariant involution $\tilde{\sigma}$ of $Q/\Sigma^+$ such that $\Sigma/\Sigma^+ \cong < \tilde{\sigma}>$ by Lemma \ref{lemma induced invol}. Moreover, the induced dimension vector 
$\tilde{d}$ on $Q/\Sigma^+$ is $\tilde{\sigma}$-compatible. Hence, we obtain 
the following description of the $\Sigma$-fixed loci:
\[ \Rep^{\Sigma} = (\Rep^{\Sigma^+})^{\Sigma/\Sigma^+} = \rep_{Q/\Si^+,\tilde{d}}^{\tilde{\sigma}}\]
and 
$\G^{\Sigma} = (\G^{\Sigma^+})^{\Sigma/\Sigma^+}= \mathbf{G}_{Q/\Sigma^+,\tilde{d}}^{\tilde{\sigma}}$.

\begin{lemma}
Let $d$ and $\theta$ be $\Sigma$-compatible. Then the following statements hold.
\begin{enumerate}
\item The $\Sigma$-action on $\Rep$ preserves the GIT (semi)stable sets $\Rep^{\chi_\theta-(s)s}$.
\item There is an induced algebraic $\Sigma$-action on the moduli spaces $\mathcal{M}^{\theta-(s)s}_{Q,d}$.
\end{enumerate}
\end{lemma}
\begin{proof}
The proof of first statement follows from Proposition \ref{prop induced action on quotient in covariant setting} for covariant automorphism groups and \cite[Lemma 2.1]{young}. The proof of the second statement then follows from the universal property of the GIT quotient.
\end{proof}

Henceforth, we assume that $d$ and $\theta$ are both $\Si$-compatible, so there is an induced $\Si$-action on $\Mod$. If we restrict to the $\Si^+$-action, then there is a morphism
\[ f_{\Sigma^+} : \Rep^{\Si^+}/\!/_{\chi_\theta} \G^{\Si^+} \lra (\Mod)^{\Si^+} \]
by Proposition \ref{prop constr f cov}, and the domain of this morphism is isomorphic to the moduli space $\cM_{Q/\Sigma^+,\tilde{d}}^{\tilde{\theta}-ss}$ by Corollary \ref{modular_interp_of_Si_fixed_reps}. Moreover, there is an induced action of the contravariant involution $\tilde{\sigma}$ of $Q/\Si^+$ on the domain of $f_{\Sigma^+}$, as both $\tilde{d}$ and $\tilde{\theta}$ are  $\tilde{\sigma}$-compatible. Hence, by Proposition \ref{prop contr invol constr of f} below, there is a morphism
\[ f_{\tilde{\sigma}} : \rep_{Q/\Si^+,\tilde{d}}^{\tilde{\si}} /\!/_{\chi_{\tilde{\theta}}} \mathbf{G}_{Q/\Si^+,\tilde{d}}^{\tilde{\si}} \lra (M_{Q/\Sigma^+,\tilde{d}}^{\tilde{\theta}-ss})^{\ov{\si}} \]
and so we can define $f_{\Sigma} :=f_{\Si^+} \circ f_{\tilde{\sigma}}$ and obtain the following result.

\begin{prop}\label{prop constr of f general auto gp}
For $\Si \subset \Aut(Q)$, there is a morphism
\[ f_{\Sigma} : \Rep^{\Si}/\!/_{\chi_\theta} \G^{\Si} \lra (\Mod)^{\Si}. \]
\end{prop}

Finally, let us prove the following result for contravariant involutions.

\begin{prop}\label{prop contr invol constr of f}
Let $\sigma$ be a contravariant involution of a quiver $Q$ and suppose $d$ and $\theta$ are $\sigma$-compatible. Then $\G^{\si}$ is a reductive group and the following are equivalent for $M \in \Rep^{\si}$.
\begin{enumerate}
\item $M$ is GIT semistable for the $\G$-action on $\Rep$ with respect to the character $\chi_\theta : \G\lra \GG_m$.
\item $M$ is GIT semistable for the $\G^{\si}$-action on $\Rep^{\si}$ with respect to the restricted character $\chi_\theta : \G^{\si} \lra \GG_m$.
\end{enumerate}
Furthermore, there is an induced morphism
$ f_\sigma : \Rep^{\si}/\!/_{\chi_\theta} \G^{\si} \lra (\Mod)^{\si}$.
\end{prop}
\begin{proof}
For the statement that $\G^{\si}$ is reductive, see \cite[$\S$2]{Derksen_Weyman} (reproved in \cite[$\S$2.2]{young}, where it is shown that $\G^{\si}$ is isomorphic to a product of orthogonal and general linear groups; if one applies a modifying family to \eqref{action_on_reps_with_sign} and \eqref{action_on_group_with_sign}, symplectic groups will appear). The proof of these equivalences follows by \cite[Proposition 2.2]{young} (or by appropriately adapting the argument in Lemma \ref{lemma compare ss}), and the construction of $f_\sigma$ follows as in Proposition \ref{prop constr f cov}.
\end{proof}

Analogously to Construction \ref{constr_type_map_def}, one can also show that for $\Sigma \subset \Aut(Q)$, there is a type map
\[ \cT: \Mods(k)^{\Si} \lra H^2(\Si,\Delta(k))\]
and define modifying families in an analogous way to Definition \ref{modifying_family_def} in order to produce a decomposition of the $\Si$-fixed locus.

For a contravariant involution $\sigma$ of $Q$, we can study the morphism $f_\sigma$, or more precisely, its restriction to the regularly stable locus
$ f_\sigma^{rs} : \cM^{(\si,\theta)-rs}_{Q,d} \lra (\Mods)^\si$, by using the group cohomology of $\ZZ/2\ZZ \cong <\!\!\sigma \! \!>$. 
One can prove a straightforward generalisation of Proposition \ref{non_empty_fibres} for the fibres of $f_{\sigma}^{rs}$; thus, $f_\sigma$ is injective, as $H^1(\ZZ/2\ZZ, k^\times) = 1$ (we recall that $\sigma$ now acts on $k^\times$ by inversion). 
One can also define modifying families, which are now given by a single element $u_\sigma \in \G(k)$ such that $u_\sigma \Psi_\sigma(u_\sigma) \in \Delta(k)$, and use these families to modify the actions $\Phi$ and $\Psi$, without changing the action on $\Mods$; as $\Z/2\Z$ acts on $k^\times$ by inversion, we have $H^2(\ZZ/2\ZZ, k^\times) = \{ \pm 1 \}$, so there are only two possible cohomology classes for the 2-cocycle associated to a modifying family. This appears in \cite{young} in different language, coming from physics: the duality structures in \emph{loc.\ cit.} correspond to our modifying families. One can also provide a decomposition of $(\Mods)^{\si} $ by varying these modifying families using the cohomology of $\ZZ/2\ZZ$. One can give a representation-theoretic description of this decomposition as in $\S$\ref{representation_theoretic_interp}, by defining notions of $(\sigma,\theta)$-(semi)stability for $(\si,c_u)$-equivariant representations. Young refers to such equivariant representations as self-dual representation (\textit{cf.}\ \cite[Theorem 2.7]{young} for an analogue of Theorem \ref{comparison_with_GIT_stability} in the contravariant setting).

One can thus provide a decomposition of $(\Mods)^{\Si}$ for an arbitrary subgroup $\Si \subset \Aut(Q)$ using the group cohomology of $\Si$; however, we do not go through the details, as it is analogous to the case where $\Si \subset \Aut^+(Q)$. 

\section{Branes}\label{branes_section}

Starting from a quiver $Q$, moduli spaces of representations of the doubled quiver $\ov{Q}$ (satisfying some relations) have a natural algebraic symplectic structure and we show that automorphisms of $\ov{Q}$ provide natural examples of Lagrangian and symplectic subvarieties. Over the complex numbers, these moduli spaces are hyperk\"{a}hler when they are smooth and we can describe the fixed locus in the language of branes \cite{Kapustin_Witten} as follows.

\begin{defn} A brane in a hyperk\"{a}hler manifold $(M,g,I,J,K,\omega_I,\omega_J,\omega_K)$ is a submanifold which is either holomorphic or Lagrangian with respect to each of the three K\"{a}hler structures on $M$.  A brane is called of type $A$ (respectively $B$) with respect to a given K\"{a}hler structure if it is Lagrangian (respectively holomorphic) for this K\"{a}hler structure. The type of the brane is encoded in a triple $T_IT_JT_K$, where $T_I=A$ or $B$ is the type for the K\"{a}hler structure $(g,I)$ and so on. 
\end{defn}

As $K=IJ$, there are 4 possible types of branes: $BBB$, $BAA$, $ABA$ and $AAB$. We will show that we can construct each type of brane as a fixed locus of an involution. The study of branes in Nakajima quiver varieties has already been initiated in \cite{fjm}, where the authors use involutions such as complex conjugation, multiplication by $-1$ and transposition, to construct different branes. In the present section, we construct new examples coming from automorphisms of the quiver.

\subsection{The algebraic case}

We assume throughout that $k$ is a field of characteristic different from $2$. 

\begin{defn}[Doubled quiver]
For a quiver $Q = (V,A,h,t)$, the doubled quiver is
$\overline{Q}=(V,\overline{A}, h,t)$ where $\overline{A} = A \cup A^*$ for $ A^* := \{ a^* : h(a) \lra t(a)\}_{a \in A}$.
\end{defn}

\noindent The central motivation for considering the doubled quiver is that 
$$\RepQbar = \Rep \times \Rep^* \cong T^*\Rep$$
is an algebraic symplectic variety, with the Liouville symplectic form $\omega$. 
Explicitly, if $M=(M_a,M_{a^*})_{a \in A}$ and $N = (N_a,N_{a^*})_{a\in A}$ are 
points in $\RepQbar$, then
\begin{equation}\label{Liouville_form} \omega(M,N) = \sum_{a \in A} \tr(M_aN_{a^*} - M_{a^*}N_a).\end{equation}

The action of $\GQbar = \G$ on $\RepQbar$ is symplectic and there is an algebraic moment map $ \mu :\RepQbar \lra \LieG^*$, 
where $\LieG $ is the Lie algebra of $\G$; explicitly, for $M \in \RepQbar$ and $B \in \LieG$ we have 
\begin{equation}\label{eqn alg mmap quiver}
\mu(M) \cdot B = \sum_{a \in A} \tr(M_{a^*} (B^{\#}_M)_a)= \sum_{a \in A} \tr(M_{a^*} (B_{h(a)}M_a - M_a B_{t(a)}))
\end{equation}
where $B^{\#}_{M}= (B_{h(a)} M_a - M_a B_{t(a)})_{a \in A}$ is the infinitesimal action of $B$ on $(M_a)_{a \in A}$. 
The moment map is a $\G$-equivariant morphism that satisfies the infinitesimal lifting property
$ d_M\mu(\eta) \cdot B = \omega(B^{\#}_M,\eta)$. 
By using the standard non-degenerate quadratic form $(B,C)\lmt \tr(^{t}BC)$ on the Lie algebra of each general linear group, we can naturally identify $\LieG \cong \LieG^*$ and view the moment map as a morphism $\mu :\RepQbar \lra \LieG$ given by
$\mu(M) = \sum_{a \in A} [M_a,M_{a^*}]$.

\begin{defn}\label{alg_symp_red}
Let $\chi$ be a character of $\G$ and let $\eta \in \LieG$ be a coadjoint fixed point; then $\G$ acts on $\mu^{-1}(\eta)$ by 
the equivariance of the moment map. The algebraic 
symplectic reduction at $(\chi,\eta)$ is the GIT quotient $\mu^{-1}(\eta) /\!/_{\chi} \G$.
\end{defn}

If the GIT semistable and stable locus on $\mu^{-1}(\eta)$ with respect to $\chi$ agree, then as $\Gbar$ acts freely on the stable locus, the symplectic reduction $\mu^{-1}(\eta)/\!/_{\chi} \G$ is a smooth algebraic symplectic variety, whose form is induced by the Liouville form on $T^*\rep_d(Q)$; 
this is an algebraic version of the Marsden--Weinstein Theorem \cite{mw} 
(for example, see \cite{ginzburg}). For $\chi= \chi_{\theta}$, the closed $\G$-invariant subvariety $\mu^{-1}(\eta) \subset \rep_{\overline{Q},d}$ determines a closed immersion
\[ \mu^{-1}(\eta)/\!/_{\chi_\theta} \G \hookrightarrow \cM^{\theta-ss}_{\overline{Q},d}.\]
Moreover, for a tuple of complex numbers $(\eta_v)_{v \in V}$, which determines an adjoint fixed point $\eta = (\eta_v \text{Id}_{d_v})_{v \in V} \in \fg$, we have that $\mu^{-1}(\eta)/\!/_{\chi_\theta} \G$ is the moduli space of $\theta$-semistable $d$-dimensional representations of $\overline{Q}$ satisfying the relations 
\[\cR_\eta = \Big\{ \sum_{a\, |\, t(a) = v} M_a M_{a^*} - \sum_{a\, |\, h(a) =v} M_{a^*}M_a = \eta_v I_{d_v} \: \: \: \forall v\in V\Big\}.\]

We now describe the fixed loci of quiver automorphism groups acting on this algebraic symplectic reduction in terms of its symplectic geometry. 

\begin{defn}\label{def anti sympl auto}
We let $\Aut_*(\overline{Q})$ denote the subgroup of $\Aut(\overline{Q})$ consisting of automorphisms $\sigma$ satisfying the conditions:
\begin{enumerate}
\item For all $a \in A$, $\sigma(a^*) = \sigma(a)^*$.
\item Either $\si(A) \subset A$ or $\si(A)\subset A^*$.
\end{enumerate}
An automorphism $\sigma\in \Aut_*(\ov{Q})$ is said to be $\ov{Q}$-symplectic if $\si(A)\subset A$, and $\ov{Q}$-anti-symplectic if $\sigma(A) \subset A^*$.
\end{defn}

\noindent Every $\sigma \in \Aut(Q)$ can be extended to a $\ov{Q}$-symplectic automorphism $\sigma \in \Aut_*(\overline{Q})$ by $\sigma(a^*) := \sigma(a)^*$. There is a canonical contravariant involution $\sigma \in \Aut_*(\overline{Q})$ which fixes all vertices and is given by $\sigma(a):=a^*$ on $a \in A$; it is $\ov{Q}$-anti-symplectic. Note that Condition (1) in Definition \ref{def anti sympl auto} does not imply Condition (2): consider for instance $$Q=
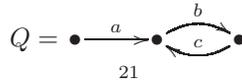
\xymatrix@1{ \bullet \ar[r]^{a} & \bullet \ar@/^/[r]^{b} & \bullet \ar@/^/[l]_{c}}$$
and the contravariant involution $\sigma \in \Aut_*(\overline{Q})$ which fixes all vertices and sends $\sigma(a) = a^*$, $\sigma(b) =c$ and $\sigma(b^*) = c^*$.

There is a group morphism $s:\Aut_*(\ov{Q})\lra \{\pm1\}$ sending $\si\in\Aut_*(\ov{Q})$ to $-1$ if and only if $\sigma$ is $\ov{Q}$-anti-symplectic. 

\begin{prop}\label{lemma sigma sympl}
Let $\sigma \in Aut_*(\ov{Q})$ and let $d$ be a $\sigma$-compatible dimension vector. Then $\si^*\omega = s(\si)\,\omega$; that is, if $\si$ satisfies Property (1) of Definition \ref{def anti sympl auto}, then Property (2) implies that $\si$ is either symplectic or anti-symplectic in the usual sense.
Moreover, for $M \in \RepQbar$ and $B \in \LieG$, we have $\mu(\sigma(M)) \cdot \sigma(B) = s(\si) (\mu(M) \cdot B)$.
\end{prop}
\begin{proof}
As $d$ is $\sigma$-compatible, there is an induced action of $\sigma$ on $\RepQbar$. 
For $M \in \RepQbar$ and let $Y,Z \in T_M \RepQbar \cong \RepQbar$, we have $d_M\sigma(Y)= (Y_{\sigma(a)})_{a \in \overline{A}}$ if $\sigma$ is covariant, and $ d_M\sigma(Y)= (^{t}Y_{\sigma(a)})_{a \in \overline{A}}$ if $\sigma$ is contravariant. As the calculations are similar in the covariant and contravariant case, we only give the details for a covariant automorphism $\sigma$, which is by assumption either $\ov{Q}$-symplectic or $\ov{Q}$-anti-symplectic. As $(\sigma^*\omega)_M(Y,Z) = \omega_{\si(M)}(d_M\si(Y),d_M\si(Z))$, we have
\begin{align*}
(\sigma^*\omega)_M(Y,Z) & = \sum_{a \in A} \tr(Y_{\sigma(a)} Z_{\sigma(a^*)} -  Y_{\sigma(a^*)}Z_{\sigma(a)}) \\
& = \sum_{\begin{smallmatrix}a \in A \,|\, \\ b:=\sigma(a) \in A\end{smallmatrix}} \tr(Y_{b} Z_{b^*} -  Y_{b^*} Z_{b}) - \sum_{\begin{smallmatrix} a \in A\, |\, \\ c:=\sigma(a)^* \in A  \end{smallmatrix}} \tr(Y_{c} Z_{c^*}-Y_{c^*} Z_{c}) \\
& = s(\si)\, \omega_M(Y,Z) 
 \end{align*}
where in the second equality we use the fact that $\sigma(a^*)=\si(a)^*$.
 
The derivative of the automorphism $\sigma$ on $\G$ induces an automorphism
\[ \sigma: \LieG \lra \LieG, \quad (B_v)_{v\in V} \lmt \left\{ \begin{array}{cl}
(B_{\sigma(v)})_{v \in V} & \text{if } \sigma \text{ is covariant}, \\
 (-^{t}B_{\sigma(v)})_{v \in V} & \text{if } \sigma \text{ is contravariant.} \end{array} \right.\]
Thus, for covariant $\sigma$ and for $M \in \RepQbar$ and $B \in \LieG$, we have
\begin{align*}
\mu(\sigma(M)) \cdot \sigma(B)  & =  \sum_{a \in A} \tr( M_{\sigma(a)} M_{\sigma(a^*)} B_{\sigma(h(a))} -B_{\sigma(t(a))}M_{\sigma(a^*)}M_{\sigma(a)}) \\
& =  \sum_{\begin{smallmatrix}a \in A \,|\,\\b=\sigma(a) \in A \end{smallmatrix}} \tr ( M_b M_{b^*} B_{h(b)} - B_{t(b)}  M_{b^*}M_{b})  \\
& \quad \quad  - \sum_{\begin{smallmatrix}a \in A \,|\, \\ c=\sigma(a)^* \in A \end{smallmatrix}} \tr ( M_c M_{c^*} B_{h(b)} - B_{t(b)} M_{c^*}M_{c})
\\
& = s(\si)(\mu(M)\cdot B)
\end{align*}
where in the second equality we use the fact that $\sigma(a^*)=\si(a)^*$.
\end{proof}

\begin{cor}\label{cor sigma on asr}
Let $\sigma \in \Aut_*(\overline{Q})$ be an involution and suppose that the dimension vector $d$ and stability parameter $\theta$ are $\sigma$-compatible and that $\eta \in \LieG^*$ is coadjoint fixed and satisfies $\sigma(\eta) = s(\sigma) \eta$. Then $\sigma$ preserves $\mu^{-1}(\eta)$ and there is an induced automorphism $\sigma: \mu^{-1}(\eta)/\!/_{\chi_\theta}\G \lra \mu^{-1}(\eta)/\!/_{\chi_\theta} \G$.
\end{cor}

We can now describe the geometry of the fixed locus of an involution.

\begin{prop}\label{Lagr fixed locus for a-sym invol}
Let $\sigma \in \Aut_*(\overline{Q})$ be an involution. If $\si$ is $\ov{Q}$-anti-symplectic, then $\RepQbar^\sigma \cong \Rep\neq\emptyset$ and this fixed locus is a Lagrangian subvariety of $\RepQbar$. If $\si$ is $\ov{Q}$-symplectic, then $\RepQbar^\si$ is a symplectic subvariety of $\RepQbar$. When the symplectic reduction $\mu^{-1}(\eta) /\!/_{\chi_\theta} \G$ is a smooth  algebraic variety, the fixed locus of the induced involution is therefore Lagrangian if $\si$ is $\ov{Q}$-anti-symplectic and symplectic if $\si$ is $\ov{Q}$-symplectic.
\end{prop}
\begin{proof}
If $\sigma$ is a $\ov{Q}$-anti-symplectic involution, no arrows are fixed by $\sigma$ and $\overline{A} = A \sqcup A^* = A \sqcup \sigma(A)$. Hence $\RepQbar^\sigma \cong \Rep$ (in particular, this has half the dimension of $\RepQbar$). The remaining statements follow from general properties of anti-symplectic and symplectic involutions of a non-singular symplectic variety $(X,\omega)$ over a field of characteristic $\neq2$: if $\si:X\lra X$ is of order $2$, the tangent space at $x$ is $T_x X = \ker(T_x\si - \mathrm{Id}) \oplus \ker(T_x\si + \mathrm{Id})$; these two subspaces are Lagrangian if $\si^*\omega=-\omega$ and each other's symplectic complement if $\si^*\omega=\omega$. Note that  one of these relations indeed holds here, in view of Proposition \ref{lemma sigma sympl}
\end{proof}

\subsection{The hyperk\"{a}hler case}

Over the complex numbers, the algebraic symplectic reduction has a hyperk\"{a}hler structure, as it can be interpreted as a hyperk\"{a}hler reduction via the Kempf--Ness theorem. 
In fact, we can generalise the above situation as follows.
Let $X = \AA^n_\CC$ be a complex affine space with a linear action of a complex reductive group $G$. 
We can assume without loss of generality that the maximal compact subgroup $U$ of $G$ acts unitarily 
(by rechoosing coordinates on $X$ if necessary). The standard Hermitian form 
$H : X \times X \lra \CC$ given by $(z,w) \lmt z^{t}\overline{w}$, 
where we consider $w$ and $z$ as row vectors, is then $U$-invariant. In particular, 
$X$ is a K\"{a}hler manifold with complex structure given by multiplication by $i \in \CC$, 
metric $g = \text{Re} H$ and symplectic form $\omega_{X} = -\text{Im} H$. 
The action of $U$ on $X$ is symplectic for $\omega_{X}$ with moment map $\mu_{X} : V \lra \fk^*$ 
given by
$ \mu_{\RR}(z) \cdot B = \frac{i}{2} H(Bz,z)$ 
where $z \in X$ and $B \in \fk^*$. The cotangent bundle is hyperk\"{a}hler, as we can identify $T^*X \cong  X \times X^* \cong \HH^n$ by $(z,\alpha) \lmt (z - \alpha j)$ and inherit the hyperk\"{a}hler structure from the quaternionic vector space $\HH^n$. Let $I$,$J$ and $K$ denote the complex structures on $T^*X$ obtained from the complex structures on $\HH^n$ given by left multiplication by $i,j,k$. We consider the associated symplectic forms 
\[ \omega_I(-,-) = g(I-,-), \quad \omega_J(-,-) = g(J-,-) \quad \text{and} \quad \omega_K(-,-) = g(K-,-).\] 
We often write $\omega_{\RR} = \omega_I$ and 
$\omega_{\CC} = \omega_J + i \omega_K$, which is the Liouville algebraic symplectic form. This linear $G$-action on $X$ lifts to a 
linear action of $G$ (and $U$) on $T^*X$; moreover, as $U$ acts unitarily on $X$, 
the $U$-action is symplectic with respect to $\omega_{\RR}$ and the 
$G$-action is symplectic with respect to $\omega_{\CC}$. The associated moment maps 
$\mu_{\RR} : T^*X \lra \fk^*$ and $\mu_{\CC} : T^*X \lra \fg^*$ are given by
$ \mu_{\RR}(z,\alpha) \cdot B = (\mu_{X}(z) - \mu_{X}(\alpha)) \cdot B$
and
$\mu_{\CC}(z,\alpha) \cdot A = \alpha(A^{\#}_z)$, 
where $A \in \fg$ and $(z,\alpha) \in T^*X$ and $A^{\#}_z$ denotes the infinitesimal 
action of $A$ on $z$. We often write $\mu_{HK} := (\mu_\RR,\mu_{\CC})$, which is a 
hyperk\"{a}hler moment map for the $U$-action on $T^*X$. Let $\chi \in \fk$ and $\eta \in \fg$ be coadjoint fixed, then $U$ acts on the level set
$\mu^{-1}_{HK}(\chi,\eta) := \mu_{\RR}^{-1}(\chi) \cap \mu_{\CC}^{-1}(\eta)$ 
and the  hyperk\"{a}hler reduction is the topological quotient $\mu^{-1}_{HK}(\chi,\eta)/U$, 
which inherits an orbifold hyperk\"{a}hler structure if $U$ acts with finite stabilisers on $\mu^{-1}_{HK}(\chi,\eta)$. By the Kempf-Ness theorem \cite{kempf_ness}, there is a homeomorphism between the hyperk\"{a}hler reduction and algebraic symplectic reduction:
$\mu^{-1}_{HK}(\chi,\eta)/U \cong \mu_{\CC}^{-1}(\eta)/\!/_{\chi} G$, 
where we consider $\chi$ as a character of $G$ by complexifying and exponentiating.

For a quiver $Q$, we can apply the above picture to $G = \G$ acting on 
$X = \Rep$. Then $U = \mathbf{U}_{Q,d} := \Pi_{v \in V} \mathbf{U}(d_v)$ and we take the Hermitian 
form on $\Rep$ given by 
$H(M,N) = \sum_{a \in A} \tr (M_a \:^{t}\overline{N}_a)$. 
The hyperk\"{a}hler metric $g$ on $\RepQbar$ is given by
\begin{equation}\label{metric} g(X,Y) = \text{Re} \Big( \sum_{a \in \overline{A}} \tr (X_a \: ^{t}\overline{Y}_a)  \Big);\end{equation}
therefore, $\omega_\RR= \omega_I$ is given by
$\omega_{\RR}(X,Y) =  \text{Im} \left( \sum_{a \in \overline{A}}  \tr (^{t}\overline{X}_a Y_a)  \right)$
and $\omega_{\CC}  = \omega_J + i \omega_K$ is the Liouville algebraic symplectic form $\omega$ described in \eqref{Liouville_form}. Moreover, $\mu_{\CC} = \mu : \RepQbar \lra \LieG^*$ is the algebraic moment map 
given by \eqref{eqn alg mmap quiver} and $\mu_{\RR} : \RepQbar \lra \fu_{Q,d}^*$ 
can be explicitly described as
\[ \mu_{\RR}(M) \cdot B = \frac{i}{2} \sum_{a \in \overline{A}} \tr (B_{h(a)}M_a \:^{t}\overline{M}_a - B_{t(a)} \:^{t}\overline{M}_{a} {M}_{a}). \]
If we use the standard identification $\fu_{Q,d} \cong \fu_{Q,d}^*$, then we can consider 
the real moment map as a map $\mu_{\RR} : \RepQbar \lra \fu_{Q,d}$ given by  
$\mu_{\RR}(M) = \frac{i}{2}\sum_{a \in \overline{A}} [M_a, ^{t}\overline{M}_a]$. 
If $\chi_\theta$-semistability coincides with $\chi_\theta$-stability on $\mu^{-1}(\eta)$, 
then we obtain an orbifold hyperk\"{a}hler structure on the algebraic variety
$\mu^{-1}(\eta)/\!/_{\chi_\theta} \G$.

We note that Nakajima quiver varieties can also be constructed in this manner, as a framed quiver variety can be realised as a unframed quiver variety for a different quiver (see \cite[p. 261]{CBmoment}). The effect of complex conjugation on Nakajima quiver varieties is studied in \cite{fjm}, 
where they show the fixed locus in the Nakajima quiver variety is an $ABA$-brane; see Corollary 3.10 in \textit{loc.\ cit}. We will study the geometry of the fixed locus of an automorphism $\sigma \in \Aut_*(\overline{Q})$ that is either $\ov{Q}$-symplectic or $\ov{Q}$-anti-symplectic. 

\begin{lemma}\label{prop comp involn on quiver with HK}
Let $\sigma \in \Aut_*(\overline{Q})$ and let $d$ be a $\sigma$-compatible dimension vector. Then the automorphism $\sigma$ of $\RepQbar$ has the following properties.
\begin{enumerate}
\item $\sigma$ is holomorphic with respect to $I$ and symplectic with respect to $\omega_I$.
\item If $\sigma$ is $\ov{Q}$-symplectic, then $\sigma$ is holomorphic with respect to $J$ and $K$ and symplectic with respect to $\omega_J$ and $\omega_K$. 
\item If $\sigma$ is $\ov{Q}$-anti-symplectic, then $\sigma$ is anti-holomorphic with respect to $J$ and $K$ and anti-symplectic with respect to $\omega_J$ and $\omega_K$.
\end{enumerate}
\end{lemma}
\begin{proof}
As complex conjugation and transposition commute, $\sigma$ is $I$-holomorphic. Since the hyperk\"{a}hler metric $g$ is preserved by $\sigma$ (\textit{cf.}\ the explicit form of the metric given in \eqref{metric}), it follows that $\sigma^*\omega_I = \omega_I$. If $M=(M_a,M_{a^*})_{a \in A} \in \RepQbar$, then $J \cdot (M_a,M_{a^*})_{a \in A} = (^{t}\overline{M}_{a^*}, - ^{t}\overline{M}_a)_{a \in A}$. Suppose that $\sigma$ is contravariant; then
\[ \sigma ( J \cdot (M_a,M_{a^*}) )= \sigma (^{t}\overline{M}_{a^*}, - ^{t}\overline{M}_a) = s(\si)
(\overline{M}_{\sigma(a^*)}, - \overline{M}_{\sigma(a)})\] (where $s(\si)=\pm 1$ depending on whether $\si$ is $\ov{Q}$-symplectic or $\ov{Q}$-anti-symplectic) and $J \cdot \sigma( M_a, M_{a^*}) = J \cdot (^{t}M_{\sigma(a)},^{t}M_{\sigma(a^*)}) = (- \overline{M}_{\sigma(a^*)}, \overline{M}_{\sigma(a)})$, which gives the compatibility of $J$ and $\sigma$. Since $K = IJ$, we can determine the compatibility of $K$ with $\sigma$ from that of $I$ and $J$. The final statements about the compatibility of $\sigma$ with $\omega_J$ and $\omega_K$ follow from Proposition \ref{lemma sigma sympl}. A very similar computation shows the result also holds when $\sigma$ is covariant.
\end{proof}

\begin{ass}\label{Ass needed for main brane thms}
Given a subgroup $\Si\subset \Aut_*(\ov{Q})$, we shall assume that the dimension vector $d$ and the stability parameter $\theta$ are $\Si$-compatible. Let $\eta \in \LieG$ be a coadjoint fixed element such that $\sigma (\eta) = s(\si) \eta$ for all $\si \in \Si$. Then there is an induced $\Sigma$-action on $\mu^{-1}(\eta) /\!/_{\chi_\theta} \G$ by Corollary \ref{cor sigma on asr}. We assume that the $\theta$-semistable locus in $\mu^{-1}(\eta)$ is non-empty and that $\Gbar$ acts freely on this $\theta$-semistable locus, so that the quotient $\mu^{-1}(\eta)/\!/_{\chi_\theta} \G$ is smooth and has a natural hyperk\"{a}hler structure.
\end{ass}

\begin{thm} \label{thm BBB brane}
Under Assumption \ref{Ass needed for main brane thms} for a subgroup $\Sigma \subset \Aut_*(\overline{Q})$ consisting of $\ov{Q}$-symplectic transformations, the $\Sigma$-fixed locus in $\mu^{-1}(\eta)/\!/_{\chi_\theta} \G$ is hyperholomorphic (or, in the language of branes, this fixed locus is a BBB-brane).
\end{thm}
\begin{proof}
If $\sigma$ is $\ov{Q}$-symplectic, then the automorphism $\sigma$ on $\RepQbar$ is holomorphic with respect to all three complex structures by Lemma \ref{prop comp involn on quiver with HK}, so the same is true for the induced automorphism on the hyperk\"{a}hler reduction with respect to its induced complex structures. Therefore, the fixed locus is a holomorphic submanifold with respect to all three complex structures.
\end{proof}

\begin{thm}\label{thm BAA brane}
Under Assumption \ref{Ass needed for main brane thms} for a $\ov{Q}$-anti-symplectic involution $\sigma$, the fixed locus of $\sigma$ acting on $\mu^{-1}(\eta)/\!/_{\chi_\theta} \G$ is a $BAA$-brane.
\end{thm}
\begin{proof}
This is a direct consequence of Lemma \ref{prop comp involn on quiver with HK} and the construction of the three symplectic structures on the hyperk\"{a}hler quotient  $\mu_{HK}^{-1}(\chi_\theta,\eta) / \mathbf{U}_{Q,d}.$
\end{proof}

If we apply Theorem \ref{thm BAA brane} to the anti-symplectic contravariant involution $\sigma$ that fixes all vertices and on arrows $a$ is given by $\sigma(a) = a^*$, then the $\sigma$-fixed locus $(\mu^{-1}(0) /\!/_{\chi_\theta} \G)^\sigma$ has as a connected component the subvariety
\[\RepQbar^\sigma /\!/_{\chi_\theta} \G^\sigma = \Rep /\!/_{\chi_\theta} \G = \Mod,\] 
which is known to be Lagrangian in $\mu^{-1}(0) /\!/_{\chi_\theta} \G$ (cf.\  \cite[Proposition 2.4]{proudfoot}).

Let $\tau \in \Gal_{\RR}$ denote complex conjugation; then we can consider compositions $\gamma := \tau \circ \sigma$, where $\sigma \in \Aut_*(\overline{Q})$ is a $\ov{Q}$-(anti)-symplectic automorphism and describe the geometry of the associated fixed loci in the language of branes. We focus on the case where $\si$ is also involutive, as this is our main source of applications.

\begin{cor}\label{cor ABA and AAB branes}
If Assumption \ref{Ass needed for main brane thms} holds for an involution $\sigma$ which commutes with $\tau$, then, when non-empty, the fixed locus of the involution $\sigma \circ \tau$ acting on the hyperk\"{a}hler manifold $\mu^{-1}(\eta)/\!/_{\chi_\theta}\G$ is
\begin{enumerate}
\item an $ABA$-brane, if $\sigma$ is $\ov{Q}$-symplectic (here we allow $\si=\mathrm{Id}$);
\item an $AAB$-brane, if $\sigma$ is $\ov{Q}$-anti-symplectic.
\end{enumerate}
\end{cor}

In particular, we see that all four types of branes ($BBB$, $BAA$, $ABA$ and $AAB$) can be constructed as the fixed locus of an involution. We also note that since $\tau$ and all $\ov{Q}$-(anti)-symplectic transformations of $\ov{Q}$ induce isometries of $\RepQbar$ by Lemma \ref{prop comp involn on quiver with HK}, the fixed loci of the various involutions we have considered are also totally geodesic submanifolds of the hyperk\"{a}hler quotient  $\mu^{-1}(\eta)/\!/_{\chi_\theta}\G$.

\section{Further applications and examples}\label{sec apps and exs}

In this section, we calculate some fixed loci for quiver group actions on the Hilbert scheme $\Hilb^n(\AA^2)$ and polygon spaces, which can both be realised as quiver moduli spaces.

\subsection{Hilbert scheme of points in the plane}

The Hilbert scheme $\Hilb^n(\AA^2)$ of $n$ points in the affine plane over an algebraically closed field $k$ can be realised as a Nakajima quiver variety associated to the Jordan quiver. More precisely, let $\ov{Q}$ be the double of the framed Jordan quiver $Q$, i.e.\ $\ov{Q}$ is the following quiver:
$$
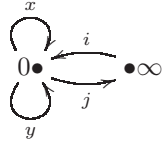
\xymatrix{0 \bullet \ar@(dl,dr)_{y} \ar@(ul,ur)^{x} \ar@/_/[r]_{j} & \bullet \infty \ar@/_/[l]_{i} }$$
where the vertex at infinity is the framing vertex. For $d = (n,1)$, we have
\[\RepQbar = \Mat_{n \times n} \times \Mat_{n \times n} \times \Mat_{n \times 1} \times \Mat_{1 \times n}\]
and for the action of $\GL_n \subset \G$ (that is, one ignores the group $\GG_m$ corresponding to the framing vertex), we have a moment map $\mu : \RepQbar \lra \fg \fl_n $ given by
\[ \mu(M_x,M_y,M_i,M_j) = [M_x,M_y] + M_i \otimes M_j,\] where $M_i$ is a column vector and $M_j$ a row vector (viewed as a linear form). The Hilbert scheme is a $\GL_n$-quotient of the level set of the moment map at zero (other level sets give rise to Calogero-Moser spaces). By \cite[Section 5.6]{ginzburg}, the affine GIT quotient of the $\GL_n$-action on $\mu^{-1}(0)$ with respect to the trivial stability parameter is isomorphic to the $n$-th symmetric power of $\AA^2$, i.e.\ $\mu^{-1}(0) /\!/ \GL_n \cong \sym^n(\AA^2)$, and the GIT quotient with respect to the character $\det : \GL_n \lra \GG_m$ (that corresponds to the stability parameter $\theta_0=-1$) is the Hilbert scheme of $n$-points on the affine plane, i.e.\ $\mu^{-1}(0) /\!/_{\det} \GL_n \cong \Hilb^n(\AA^2)$, which is a geometric quotient of $\mu^{-1}(0)^{\det-ss}=\mu^{-1}(0)^{\det-s}$, and moreover, the natural morphism from the former to the latter is the Hilbert-Chow morphism, which is an algebraic symplectic resolution of singularities. More precisely, $M := (M_x,M_y,M_i,M_j) \in \mu^{-1}(0)$ is GIT stable for the character $\det$ if $[M_x,M_y] = 0$, and $M_j = 0$, and $M_i$ is a cyclic vector for $M_x$ and $M_y$. In this case, the corresponding point in the Hilbert scheme is given by the ideal $J_M := \{ f(x,y) \in k[x,y]: f(M_x,M_y)M_i = 0 \} \subset k[x,y]$, which has codimension $n$, as $M_i$ is cyclic. We note that the ideal $J_M$ is constant on the $\GL_n$-orbit of $M$. 

Note that the present setting is slightly different from the one in Section \ref{branes_section}, insofar as one does not consider the action of the full $\G=\GL_n \times \GG_m$ on $\RepQbar$, but only that of the subgroup $\GL_n\subset \G$ (the automorphism group of the unframed quiver) and consequently, there is no global stabiliser group $\Delta$. In particular, $\GL_n$ acts freely on $\mu^{-1}(0)^{\det-s}$. For our results, this simply means that one should replace $\Delta$ with the trivial group and $\G$ with $\GL_n$.

Let us study the automorphism group of $\ov{Q}$. For reasons of valency, every automorphism of $\ov{Q}$ must fix each vertex. We have that $\Aut(\ov{Q}) = \ZZ/2\ZZ \times \ZZ/2\ZZ$, where the non-trivial automorphisms are $\si:x\lmt y$ leaving $i$ and $j$ fixed, then $\si':i\lmt j$ leaving $x$ and $y$ fixed and finally $\si\circ\si'$,
where $\sigma$ is covariant, and $\sigma'$, $\sigma \circ \sigma'$ are contravariant. The dimension vector $d = (n,1)$ is $\Aut(\ov{Q})$-compatible and every stability parameter is compatible with the covariant automorphisms, but only the trivial notion of stability is compatible with the contravariant involutions (so there is an induced action of the contravariant involutions on $\sym^n(\AA^2)$, but not on $\Hilb^n(\AA^2)$). So let us focus on the covariant involution $\sigma$; then 
$\RepQbar^\sigma = \{ M = (M_x,M_y, M_i,M_j) : M_x = M_y \}$.
Suppose that $M \in \mu^{-1}(0)^{\det-s}$ is $\sigma$-fixed; then $M = (M_x,M_x,M_i,0)$ where $M_i$ is a cyclic vector for $M_x$. In this case, the corresponding ideal $J_M \subset k[x,y]$ contains the ideal $I = (x-y)$, and so we have closed embeddings
$\spec k[x,y]/J_M \hookrightarrow \spec k[x,y]/I \hookrightarrow \spec k[x,y] = \AA^2$,
where the final morphism can be viewed as the diagonal embedding $D : \AA^1 \lra \AA^2$. Hence, the $n$ points corresponding to the ideal $J_M$ all lie on the diagonal line $D(\AA^1) \subset \AA^2$. This determines a map $(\mu^{-1}(0)^{\det-s})^\si\lra \Hilb^n(D(\AA^1)) \cong \sym^n(\AA^1) \cong \AA^n$.
As $\sigma$ acts trivially on all vertices, we have that $\GL_n^\sigma = \GL_n$. 

\begin{lemma}\label{Hilbert_scheme_of_pts_in_the_diagonal}
There is an isomorphism $(\mu^{-1}(0)^{\det-s})^\si / \GL_n \cong \Hilb^n(D(\AA^1))$.
\end{lemma}
\begin{proof}
Since the map $ (\mu^{-1}(0)^{\det-s})^\si\lra \Hilb^n(D(\AA^1))$ described above is $\GL_n$-invariant, it descends to a morphism
$(\mu^{-1}(0)^{\det-s})^\si / \GL_n \lra \Hilb^n(D(\AA^1))$
by the universal property of the GIT quotient. To show this is an isomorphism, we need to describe the inverse map. A codimension $n$ ideal $J \subset k[x,y]/I \cong k[x]$ determines a $n$-dimensional $k$-vector space $V = k[x]/J$, and multiplication by $x$ induces an endomorphism $M_x : V \lra V$ and the inclusion of the multiplicative unit induces a map $M_i : k \lra V$. Furthermore, the image of $M_i$ under repeated applications of $M_x$ cyclically generates $V$. Hence, $J = J_M$ for the $\sigma$-fixed stable point $M = (\varphi \circ M_x \circ \varphi^{-1} ,\varphi \circ M_x \circ \varphi^{-1}, \varphi \circ M_i,0)$, where $\varphi : V \lra k^n$ is a chosen isomorphism (we note that different choices of $\varphi$ correspond to different points in the $\GL_n$-orbit of $M$).
\end{proof}

By Proposition \ref{prop constr f cov} and Lemma \ref{Hilbert_scheme_of_pts_in_the_diagonal}, there is a map $f_\sigma: \Hilb^n(D(\AA^1)) \lra \Hilb^n(\AA^2)^\si$, which is injective by Proposition \ref{prop fSigma inj}, as $\sigma$ fixes all vertices of $Q$. Let us decompose $\Hilb^n(\AA^2)^\si$ by using the type map. First, we note that there is only one fibre of the type map, as $H^2(\ZZ/2\ZZ, \{1\}) = 1$. By Theorem \ref{components_of_fibres_of_the_type_map}, the trivial fibre of the type map has a decomposition indexed by $H^1(\ZZ/2\ZZ, \GL_n(k))/\{1\}$, where the $\ZZ/2\ZZ$-action on $\GL_n(k)$ is trivial. We note that $H^1(\ZZ/2\ZZ, \GL_n(k))$ is in bijection with the set of conjugacy classes of $n \times n$-matrices of order 2. The minimum polynomial of such a matrix divides $x^2 -1$, and so this matrix is diagonalisable with eigenvalues equal to $\pm 1$. Therefore, $H^1(\ZZ/2\ZZ, \GL_n(k)) \cong \{ u_0, \dots , u_n \}$,
where $u_r$ is the diagonal $n \times n$-matrix with $-1$ appearing $r$ times on the diagonal followed by $1$ appearing $n -r$ times. The element $u_0 = I_n$ corresponds to the trivial modifying family, which does not alter the action. For $r > 0$, we have
\[ {_{u_r}}\RepQbar^{\si} = \{ M = (M_x,M_y,M_i,M_j) : M_x = M_y, \: u_r M_i =M_i \text{ and } M_j u_r = M_j \}, \]
thus, $(M_i)_l = 0 = (M_j)_l$ for $1 \leq l \leq r$. In particular, for $u_n = - I_n$, the intersection of this fixed locus with $\mu^{-1}(0)^{\det-s}$ is trivial, as for $M$ to lie in this fixed locus, we must have $M_i = 0$, which cannot be a cyclic vector. Moreover, ${_{u_r}}\GL_n^{\sigma} $ is the centraliser of $u_r$ in $\GL_n$, and so ${_{u_r}}\GL_n^{\sigma} \simeq \GL_r \times \GL_{n-r}$. Hence, we have a decomposition into varieties
\[ \big(\Hilb^n(\AA^2)\big)^\sigma \simeq \bigsqcup_{ r= 0}^n ({_{u_r}}\RepQbar^{\si} \cap \mu^{-1}(0)^{\det-s}) / (\GL_r \times \GL_{n-r}). \]
This fixed locus is not a brane (note that $\sigma$ is neither $\ov{Q}$-symplectic or $\ov{Q}$-anti-symplectic because $A = \{x,i\}$ and $A^* = \{y,j\}$ so $\sigma(A) \nsubseteq A$ and $\sigma(A) \nsubseteq A^*$).

\begin{rmk} 
One could also consider the $\Gamma$-equivariant Hilbert schemes of $n$-points in the plane, for a finite group $\Gamma \subset \SL_2$. The McKay correspondence associates to such a finite group $\Gamma$ (up to conjugacy) an affine Dynkin graph of ADE type and the $\Gamma$-equivariant Hilbert scheme is a Nakajima quiver variety associated to this affine Dynkin graph. An interesting question, which we do not pursue here, is whether the fixed loci for subgroups of the automorphism group of these quivers have a special representation-theoretic interpretation. For affine Dynkin diagrams of type A, the work of Henderson and Licata \cite{Henderson_Licata} shows this to be the case.
\end{rmk} 

\subsection{Moduli of points on $\PP^1$ and polygon spaces}

For $n > 1$, fix a tuple of positive integer weights $r = (r_1, \dots , r_n)$ and an algebraically closed field $k$. Let $\cM^{r-ss}_{\PP^1}$ be the moduli space of ${r}$-semistable $n$ ordered points on $\PP^1$ (over $k$) modulo the automorphisms of the projective line, where $n$ ordered points $(p_1, \dots , p_n)$ on $\PP^1$ are ${r}$-semistable if, for all $p_0 \in \PP^1$, 
$\sum_{i\ |\ p_i = p_0} r_i \leq \sum_{i\ |\ p_i \neq p_0} r_i$. This moduli space can be constructed via GIT as a quotient of the $\SL_2$-action on $(\PP^1)^n$ with respect to an ample linearisation $\cL_r$ associated to the weights. Over the complex numbers, via the Kempf-Ness Theorem, the moduli space $\cM^{r-ss}_{\PP^1}(n)$ is homeomorphic to the polygon space $\cM_{\text{poly}}(n,r)$ consisting of $n$-gons in the Euclidean space $\EE^3$ with lengths given by $r$ modulo orientation-preserving isometries \cite{Gelfand_Macpherson,Hausmann_Knutson,KT}. 

The moduli space $\cM^{{r}-ss}_{\PP^1}(n)$ is isomorphic to a moduli space of representations for the star-shaped quiver $Q_n$ with one central vertex $v_0$ and $n$ outer vertices $v_0, \dots , v_n$ and arrows $a_i : v_i \lra v_0$ 
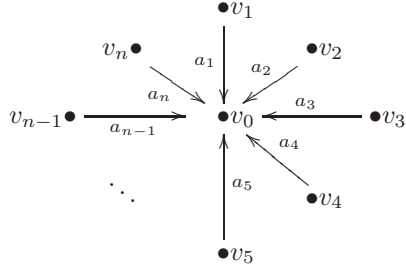
\[\xymatrixcolsep{0.3pc}\xymatrixrowsep{0.3pc}\xymatrix{ & & & \: \: \: \: \bullet v_1 \ar[ddd]_{a_1}& & & 
\\ & v_n \bullet \ar[rrdd]_{a_n} & & & &\bullet v_2 \ar[lldd]_{a_2} & \\ \\ v_{n-1} \bullet \ar[rrr]_{a_{n-1}} & & & \: \: \: \: \bullet v_0 & & & \bullet v_3 \ar[lll]_{a_3} \\  \\
& \ddots & & & & \bullet v_4 \ar[lluu]_{a_4} & \\ & & & \: \: \: \:\bullet v_5 \ar[uuu]_{a_5} & & & }\]
with dimension vector $d=(2,1,\cdots,1) \in \NN^{n+1}$, which we will often suppress from the notation, and stability parameter $\theta_r:=(- \sum_{i=1}^n r_i, 2r_1, 2r_2, \cdots , 2r_n)$. Hence, $\rep_{Q_n,d} \cong (\AA^{2})^n$ and a point $M \in \rep_{Q_n,d}$ is a tuple $M=(M_1, \cdots , M_n) $ of $k$-linear maps $M_i: k \lra k^2$; we note that the injectivity of all of these maps $M_i$ is a necessary condition for semistability for any tuple of positive weights $r$. In particular, a semistable point $M$ determines $n$ ordered points in $\PP^1\simeq \A^2/\!/_{\det_{r_i}} \Gm$ by taking the lines in $k^2$ given by the images of the injective linear maps $M_i$. Since $\mathbf{G}_{Q_n,d} = \GL_2 \times \GG_m^n$, it follows that
\[ \cM^{\theta_r-ss}_{Q_n,d} := (\AA^2)^n /\!/_{\chi_{\theta_r}} (\GL_2\times \GG_m^n ) \cong (\PP^1)^n /\!/_{\cL_r} \GL_2 =:  \cM^{r-ss}_{\PP^1}(n). \]

\noindent Every automorphism of $Q_n$ must fix $v_0$, so 
$\Aut(Q_n) = \Aut^+(Q_n)\cong S_n$. The dimension vector $d$ is $\Aut(Q_n)$-compatible, but the stability parameter $\theta_r$ is only $\Aut(Q_n)$-compatible in the equilateral case, where all weights $r_i$ coincide. By restricting our attention to subgroups $\Sigma \subset \Aut(Q_n)$, there are more weight vectors which are $\Sigma$-compatible. For instance, given any subset $I \subset \{ 1, \cdots , n\}$ of size $1 \leq m=| I | \leq n$, one can consider two remarkable subgroups of $\Aut(Q_n)$: a cyclic group $\Sigma_I \cong \ZZ/m\ZZ$ and a symmetric group $\Sigma_I'\cong S_m$, permuting the vertices indexed by $I$. For $\theta_r$ to be $\Sigma_I$-compatible (or, equivalently, $\Sigma_{I}'$ 
-compatible), we need the weight vector $r$ to satisfy $r_i = r_j$ for all $i,j$ in $I$. The quotient quivers for both $\Sigma_I$ and $\Sigma'_I$ agree: $Q_n/\Sigma_I = Q_n/\Sigma_I' \simeq Q_{n-m+1}$.
If, for notational simplicity, we suppose $I = \{1, \cdots , m\}$, then the induced stability parameter $\tilde{\theta_r} = (- \sum_{i=1}^n r_i, 2mr_1, 2r_{m+1}, \cdots, 2r_n)$ on $Q_{n-m+1}$ is the stability parameter associated to the weight vector $r_I = (mr_1, r_{m+1}, \cdots, r_n)$. By Proposition \ref{prop constr f cov}, there are morphisms
$ f_{\Sigma_I^{(')}} : \cM^{\theta_{r_I}-ss}_{Q_{n-m+1}} = \rep_{Q_n,d}^{\Sigma_I^{(')}} /\!/_{\chi_{\theta_r}} \mathbb{G}_{Q_n}^{\Sigma_I^{(')}} \lra (\cM_{Q_n}^{\theta_r-ss})^{\Si_I^{(')}}$,
whose restriction to the regularly stable locus is injective by Proposition \ref{prop fSigma inj}, as the vertex $v_0$ is fixed by every automorphism of $Q_n$. Over the complex numbers, $f_{\Sigma_I^{(')}}^{rs}$ can be identified with the inclusion $\cM_{\text{poly}}(n-m+1,r_I) \hookrightarrow \cM_{\text{poly}}(n,r)$. The fibres of the type map can be decomposed using Theorem \ref{components_of_fibres_of_the_type_map}. 

For example, consider $\sigma = (1 2)$ and fix a weight vector $r$ with $r_1 = r_2$. The type map has only one fibre, as  $H^2(<\!\si\!>, \Delta(k)) = 1$. Moreover, by Theorem \ref{components_of_fibres_of_the_type_map}, the components of $(\cM_{Q_n}^{\theta_r-s})^\si$ are indexed by $H^1(\Si,\mathbf{G}_{Q_n,d}(k)) / H^1(\Si,\Delta(k))$, where $\Sigma:=<\!\!\si\!\!>$. Note first that $H^1(\Si,\Delta(k))\simeq \{\pm 1\}$. Let $\beta : \Sigma \lra \mathbf{G}_{Q_n,d}(k)$ be a normalised 1-cocycle; then $\beta$ is given by $\beta(\sigma) = (u_0, u_1, \cdots , u_n) \in \mathbf{G}_{Q_n,d}(k)$ such that $u_0^2 = I_2$, $u_1u_2 =1$ and $u_i^2 = 1$ for $i=3, \cdots n$. Two such cocycles $\beta$ and $\beta'$ are cohomologous if there exists $g \in \mathbf{G}_{Q_n,d}(k)$ such that $g\beta(\sigma) \Psi_\sigma(g^{-1}) = \beta'(\sigma)$; that is, if and only if $g_0u_0g_0^{-1} = u_0'$ and $g_1u_1 g_2^{-1} = u_1'$ and $u_i = u_i'$ for $ i = 3, \cdots , n$. Thus, to describe the cohomology class, we can assume that $u_0$ is in Jordan normal form and $u_1 = u_2 = 1$. Since $u_0^2 = I_2$, we can assume that $u_0 \in \{\pm I_2,A:= \diag(-1,1)\}$. Hence, there is a bijection between $H^1(\Sigma, \mathbf{G}_{Q_n,d}(k)) / H^1(\Sigma, \Delta(k)) $ and the set
\[ \cB:= \{(u_0, \dots ,u_n) : u_0 \in \{I_2, A \}, \: u_1 = u_2 = 1, \: u_i = \pm 1 \text{ for } i=3,\cdots , n \},\]
which gives $2^{n-1}$ possible components of the fixed locus, although it turns out that some may be empty, as we shall now see.

\textit{Case 1.} Suppose $u_0 = I_2$. If there exists $3 \leq i \leq n$ with $u_i = -1$ , then for $M = (M_1, \dots , M_n) \in \rep_{Q_n,d}$ to lie in the fixed locus for the modified action given by $\Phi^u$, we need $M_i = 0$; in particular, the intersection of this fixed locus with the semistable set is always empty and so such a modifying family $u$ gives an empty contribution. The only non-empty contribution with $u_0 = I_2$ is the trivial element $u = 1$ in $\mathbf{G}_{Q_n,d}$, whose modified action is the original action $\Phi$.

\textit{Case 2.} Suppose $u_0 = A$. Then $(M_1, \cdots, M_n) \in \rep_{Q_n,d}$ is fixed for the modified action defined by $u$, if $M_2 = A M_1$ and, for $3 \leq i \leq n$, the image of $M_i$ is contained in the span of $(0,1) \in k^2$ (resp.\ $(1,0)  \in k^2$) if $u_i =1$ (resp.\ $u_i = -1$). Let $M \in {_u}\rep_{Q_n,d}^{\Sigma}$ be $\theta_r$-semistable; then each $M_i$ corresponds to a point $p_i\in\PP^1$, and, for $3 \leq i \leq n$, we have $p_i = [0:1]$ if $u_i = 1$ and $p_i= [1:0]$ if $u_i= -1$. As $M_2 = A M_1$, we have either $p_1 = p_2 = [0:1]$ or $p_1=[1:a]$ and $p_2 =[1:-a]$ for $a \in k$. 
Moreover, ${_u}\G^{\Si} \cong \GG_m^{n+2}$ and the image of ${_u}f_{\Si}$ in $\cM^{r-ss}_{\PP^1}$ is contained in the locus where for $i \geq 3$, we have $p_i = [0:1]$ (resp.\ $[1:0]$) if $u_i = 1$ (resp.\ $-1$). In terms of polygons, if we identify $[1:0]$ and $[0:1]$ with $(0,0,\pm1) \in \RR^3$, then these configurations are such that all but two arrows point in the $z$-direction (either up or down depending on the sign of $u_i$).

\begin{rmk}
By considering the doubled quiver $\overline{Q_n}$, one can construct  hyperk\"{a}hler analogues of the polygon space $\cM_{\text{poly}}(n,r)$, known as hyperpolygon spaces $\cH_{\text{poly}}(n,r)$ \cite{Konno}. In this case, we have $\Aut(\overline{Q_n}) \cong S_n \times \ZZ/2\ZZ$, where $\Aut^+(\overline{Q_n})=S_n$ and we can consider the contravariant involution $\sigma$ which fixed all vertices and sends $a_i$ to $a_i^*$ as a generator of $\ZZ/2\ZZ$. All covariant (resp.\ contravariant) automorphisms are $\overline{Q_n}$-symplectic (resp.\ $\overline{Q_n}$-anti-symplectic). In particular, our results on branes in $\S$\ref{branes_section} imply that for any permutation $\gamma \in S_n$ of order 2, we have that 
\begin{enumerate}
\item $\cH_{\text{poly}}(n,r)^{\gamma}$ is a $BBB$-brane,
\item $\cH_{\text{poly}}(n,r)^{\sigma \circ \gamma}$ is a $BAA$-brane,
\item $\cH_{\text{poly}}(n,r)^{\gamma \circ \tau}$ is a $ABA$-brane,
\item\ $\cH_{\text{poly}}(n,r)^{\sigma \circ \gamma \circ \tau}$ is a $AAB$-brane,
\end{enumerate}
where $\tau$ denotes complex conjugation. 
\end{rmk}

\appendix

\section{Relevant results from group cohomology}\label{sec gp coh}

Let us collect some results concerning the group cohomology of a finite group $\Si$. For a $\ZZ \Si$-module $A$, the cohomology groups of $\Si$ with $A$-coefficients are defined by
\[ H^n(\Si,A):= \Ext_{\ZZ \Si}^n(\ZZ,\Si) \] 
for all $n \geq 0$. These are functorial in $\Si$ and $A$ with associated long exact sequences; for a detailed introduction to group cohomology with abelian coefficients, see \cite{Brown}. 

\subsection{Coefficients in a finitely generated abelian group}

Let $\Si$ be a finite group and $A$ be a $\ZZ \Si$-module whose underlying abelian group is finitely generated; the groups $H^i(\Si,A)$ can be computed using the bar resolution of $\ZZ$, which is a resolution by free $\ZZ \Si$-module where in degree $n$ one has the free $\ZZ$-module $\ZZ[\Si^{n+1}]$ generated by the elements of $\Si^{n+1}$ (see \cite[Chapter I.5]{Brown}). Therefore
\[ \Hom_{\ZZ \Si} (\ZZ[\Si^{n+1}], A) \cong \Map(\Si^n,A)\]
is a finitely generated group, because $A$ is finitely generated and $\Si^n$ is finite. Since $H^n(\Si,A)$ are computed as the cohomology groups of the complex $\Hom_{\ZZ \Si} (\ZZ[\Si^{\bullet}], A)$, the groups  $H^n(\Si,A)$ are also finitely generated. As a consequence of Shapiro's Lemma \cite[Proposition III.6.2]{Brown}, we have the following result.

\begin{prop} \cite[Corollary III.10.2]{Brown}
For a finite group $\Si$ and a $\ZZ \Si$-module $A$, the group $H^n(G,A)$ is annihilated by $|G|$ for all $n > 0$.
\end{prop}

By combining the above results, we have the following well-known result.

\begin{cor}\label{cor finite for coeff in fg ab gp}
Let $\Si$ be a finite group and $A$ be a $\ZZ \Si$-module whose underlying abelian group is finitely generated; then $H^n(\Si,A)$ is finite for $n >0$.
\end{cor}

In particular, for $A = \ZZ$, we have $H^n(\Si,\ZZ) = 0$ for $n >0$.

\subsection{Trivial coefficients in the complex multiplicative group}

\begin{prop} \cite[(4.5)]{Webb}
Let $\Si$ be a finite group with trivial action on $\CC^*$; then we have isomorphisms for $i > 0$
\[ H^i(\Si,\CC^*) \cong H^i(\Si,\QQ/\ZZ) \cong H^{i+1}(\Si,\ZZ). \]
\end{prop}

By combining this with Corollary \ref{cor finite for coeff in fg ab gp}, we obtain the following result.

\begin{cor}\label{cor trivial coeff cstar}
Let $\Si$ be a finite group acting trivially on $\CC^*$; then $H^i(\Si,\CC^*)$ is finite for all $i > 0$.
\end{cor}

\subsection{First cohomology with non-abelian coefficients}

Let us recall how to define the first (and $0$th) cohomology groups of $\Si$ with coefficients in a non-abelian group $G$, which has a $\Si$-action given by a group homomorphism $\varphi : \Si \lra \Aut(G)$. 

\begin{defn}
A 1-cocycle on $\Si$ with $G$-coefficients is a function $f : \Si \lra G$ such that $f(s_1s_2) = f(s_1)\varphi_{s_1}(f(s_2))$ for all $s_1, s_2 \in G$. We let $Z^1_\varphi(\Si,G)$ denote the space of $1$-cocycles. We say two 1-cocycles $f,f' \in Z^1_\varphi(\Si,G)$ are cohomologous if there exists $g \in G$ such that $gf(s) = f'(s) \varphi_s(g)$ for all $s \in \Si$; this defines an equivalence relation $\sim$ on $Z^1_\varphi(\Si,G)$. We define $H^0_\varphi(\Si,G) := H^G$ and $H^1_\varphi(\Si,G):=Z^1_\varphi(\Si,G)/\sim$.
\end{defn}

We note that $H^1_\varphi(\Si,G)$ is naturally a pointed set with the point given by the trivial 1-cocycle, which sends every element in $\Si$ to the identity in $G$. Often we drop the subscript $\varphi$ denoting the action.

The group cohomology with non-abelian coefficients also has some functorial properties and associated exact sequences, which terminate when the relevant cohomologies are no longer definable. Let $1 \ra G' \ra G \ra G'' \ra 1$ be an exact sequence of groups such that $G'$ is abelian and central in $G$. Suppose that $\Si$ acts on all these groups and this short exact sequence is $\Si$-equivariant; then there is an exact sequence 
\[ 1 \ra (G')^{\Si} \ra G^{\Si} \ra (G'')^{\Si} \ra H^1(\Si,G) \ra H^1(\Si,G) \ra H^1(\Si,G) \ra H^2(\Si,G') \]
which terminates at this point (if $G$ is non-abelian). In this setting, the group $H^1(\Si,G')$ acts on the pointed set $H^1(\Si,G)$: for $f' \in Z(\Si,G')$ and $f \in Z(\Si,G)$ we define $f' \cdot f'' $ by $(f' \cdot f)(s) = f'(s) \cdot f(s)$, where this product is just the multiplication in $G$. As $G'  \subset G$ is central, this is a 1-cocycle. We note that this is not the same as the cokernel of the map $H^1(\Si,G') \ra H^1(\Si,G)$ in the category of pointed sets.

\subsection{Group cohomology arising in the context of quiver automorphisms}

Let $\Si \subset \Aut(Q)$ be a subgroup of automorphisms of a quiver $Q$; then there is an induced $\Si$-action on $\G$ as defined at \eqref{cov action}, and the subgroup $\Delta \cong \GG_m$ is preserved by that $\Si$-action. In this section, we give some sufficient conditions for the finiteness of the group cohomology appearing in Theorem \ref{decomp_thm_qaut_intro}.

\begin{prop}\label{prop gp coh with coeffs in Delta over C}
Let $\Si \subset \Aut(Q)$ be a subgroup acting on $\Delta(\CC) \cong \CC^*$ as above.
\begin{enumerate}
\item If $\Si \subset \Aut^+(Q)$, then this $\Si$-action on  $\Delta(\CC)$ is trivial and $H^i(\Si,\Delta(\CC))$ is finite for all $i > 0$.
\item  if $\Si = \langle \sigma \rangle$ for a contravariant involution $\sigma$, then $\Si$-acts on $\Delta$ by inversion and $H^1(\Si,\Delta(\CC))$ is trivial and $H^2(\Si,\Delta(\CC)) = \ZZ/2\ZZ$.
\end{enumerate}
\end{prop}
\begin{proof}
We identified the $\Si$-action on $\Delta$ in $\S$\ref{sec action quiver autos}. The first claim follows from Corollary \ref{cor trivial coeff cstar} and the second claim follows from the computation of these groups at the end of $\S$\ref{sec cov and contr case}.
\end{proof}

\begin{thm}
Let $k = \CC$ and let $\Si \subset \Aut(Q)$ be a subgroup of automorphisms of a quiver $Q$; then for the above induced actions of $\Si$ on $\G$ and $\Delta$, we have that $H^1(\Si, \G(\CC))$ and $H^i(\Si,\Delta(\CC))$ for $ i = 1,2$ are finite sets.
\end{thm}
\begin{proof}
The finiteness for $H^1$ is proved in \cite[Corollary 1.1]{AnWang}. 
The finiteness of $H^2(\Si,\Delta(\CC))$ follows from Proposition \ref{prop gp coh with coeffs in Delta over C}.
\end{proof}

Let us consider the group cohomology for $\Sigma$ a cyclic group. For a cyclic group with abelian coefficients, the group cohomology is well-understood: for an algebraically closed field $k$ with trivial $\Sigma$-action on $k^*$, we have 
\[ H^1(\ZZ/n\ZZ,k^*) =\mu_n(k) \quad \text{and} \quad H^2(\ZZ/n\ZZ,k^*) = 1\]
where the second group is described in Example \ref{ex H2 cyclic gp}. It remains to describe the first cohomology with $\G(k)$-coefficients.

\begin{prop} 
Let $k$ be an algebraically closed field and $\Sigma  \subset \Aut^+(Q)$ be a cyclic group of covariant quiver automorphisms. Then $H^1(\Sigma, \G(k))$ is finite.
\end{prop}
\begin{proof}
As the $\Sigma$-action on $\G(k)$ comes from the $\Sigma$-action on $V$, we have
\[ H^1(\Sigma, \G(k)) = \prod_{\Sigma \cdot v \in V/\Sigma} H^1(\Sigma, \prod_{v' \in \Sigma \cdot v} \GL_{d_{v'}}(k)) \]
and it suffices to check the sets appearing in the right hand side are all finite. 

We claim that the set $H^1 (\ZZ/n\ZZ, \GL_{d}(k))$ is finite for the trivial action. Indeed, as the action is trivial, a 1-cocycle is just a group homomorphism and two 1-cocycles are cohomologous if and only if they are conjugate. Since $\ZZ/n\ZZ$ is cyclic, such a group homomorphism is determined by the image of the generator and, as we are considering it up to conjugation over an algebraically closed field, we can assume that this matrix is in Jordan normal form and all the eigenvalues are $n$th roots of unity. Hence, $H^1 (\ZZ/n\ZZ, \GL_{d}(k))$ is a subset of the set of Jordan normal forms in $\GL_{d}(k)$ whose eigenvalues are all $n$th roots of unity in $k$, which is finite.

For an orbit $\Sigma \cdot v$, we note that the group $\ov{\Sigma}_v:=\Sigma /\Stab_{\Sigma}(v)$ acts freely and transitively on $\Sigma \cdot v$ and, as $d$ is $\Sigma$-compatible, we have $d_{v'} = d_v$ for all $v \in V$. By considering the associated exact sequence, it suffices to prove that both
\[ H^1(\ov{\Sigma}_v,  \prod_{v' \in \Sigma \cdot v} \GL_{d_v})\quad \text{and}\quad  H^1(\Stab_{\Sigma}(v),  \prod_{v' \in \Sigma \cdot v} \GL_{d_v})\]
are finite. The finiteness of the second group follows as $H^1(\ZZ/n\ZZ, \GL_{d}(k))$ is finite. 

For the finiteness of the first group, we note that the $\ov{\Sigma}_v$-action on $\prod_{v' \in \Sigma \cdot v} \GL_{d_v}$ is induced by the trivial action of the trivial subgroup $1$ on $\GL_{d_v}$ and by a version of Shapiro's Lemma for non-abelian coefficients (see \cite[Proposition 8]{Stix}), we have
\[ H^1(\ov{\Sigma}_v,  \prod_{v' \in \Sigma \cdot v} \GL_{d_v}(k)) \cong H^1 (1, \GL_{d_v}(k))=1. \]
This completes the proof.
\end{proof}

In conclusion, we obtain the following result.

\begin{cor}\label{cor index finite}
Let $k$ be an algebraically closed field and $\Sigma \subset \Aut^+(Q)$. Suppose that either of the following statements holds
\begin{enumerate}
\renewcommand{\labelenumi}{(\roman{enumi})}
\item $k=\CC$ and $\Sigma$ is arbitrary,
\item $k$ is of arbitrary characteristic and $\Sigma$ is cyclic.
\end{enumerate}
Then $H^1(\Si, \G(k))$ and $H^i(\Si,\Delta(k))$ for $ i = 1,2$ are finite sets.
\end{cor}


\def\cprime{$'$}

\end{document}